\documentclass{article}

\usepackage[final]{neurips_2019}

\usepackage[utf8]{inputenc} \usepackage[T1]{fontenc}    \usepackage{hyperref}       \usepackage{url}            \usepackage{booktabs}       \usepackage{amsfonts}       \usepackage{nicefrac}       \usepackage{microtype}      \usepackage{fancybox}

\usepackage{pdfpages}
\usepackage{microtype}
\usepackage{graphicx}
\usepackage{subfigure}

\usepackage{amssymb}
\usepackage{amsmath}
\usepackage{amsthm}
\usepackage{mathtools}
\usepackage{subfiles}

\usepackage{multirow}
\usepackage{makecell}
\usepackage{paralist}
\usepackage{url}

\usepackage{algorithm}
\usepackage{algpseudocode}
\algnewcommand{\IfOneLine}[2]{
  \State \algorithmicif\ #1\ \algorithmicthen\ #2}
\newcommand{\SOSP}{SOSP }
\newcommand{\SOSPs}{SOSPs }
\newtheorem{definition}{Definition}
\newtheorem{lemma}{Lemma}

\newtheorem{theorem}{Theorem}

\newtheorem{claim}{Claim}

\newtheorem*{lemma*}{Lemma}
\newtheorem*{theorem*}{Theorem}

\DeclarePairedDelimiter{\abs}{\lvert}{\rvert}
\DeclarePairedDelimiter{\norm}{\lVert}{\rVert}

\DeclareMathOperator*{\argmin}{arg\,min}

\title{Efficiently avoiding saddle points\\ with zero order methods: No gradients required 
}

\author{Lampros Flokas\thanks{Equal contribution} \\
  Department of Computer Science\\
  Columbia University\\
  New York, NY 10025 \\
  \texttt{lamflokas@cs.columbia.edu} \\
\And
  Emmanouil V. Vlatakis-Gkaragkounis\footnotemark[1] \\
  Department of Computer Science\\
  Columbia University\\
  New York, NY 10025 \\
  \texttt{emvlatakis@cs.columbia.edu} \\
  \AND
  Georgios Piliouras \\
  Engineering Systems and Design  \\
  Singapore University of Technology and Design\\
  Singapore \\
  \texttt{georgios@sutd.edu.sg} \\
}

\begin{document}

\maketitle

\begin{abstract}
   We consider the case of derivative-free algorithms for non-convex optimization, also known as zero order algorithms, that use only function evaluations rather than gradients. For a wide variety of gradient approximators based on finite differences, we establish asymptotic convergence to second order stationary points using a carefully tailored application of the Stable Manifold Theorem.  Regarding efficiency, we introduce a noisy zero-order method that converges to second order stationary points, i.e avoids saddle points. Our algorithm uses only $\tilde{\mathcal{O}}(1 / \epsilon^2)$ approximate gradient calculations and, thus, it matches the converge rate guarantees of their exact gradient counterparts up to constants. In contrast to previous work, our convergence rate analysis avoids imposing additional dimension dependent slowdowns in the number of iterations required for non-convex zero order optimization.
\end{abstract}

\section{Introduction}

Given a function $f: \mathbb{R}^d \to \mathbb{R}$, solving the problem
\begin{equation*}
    \mathbf{x}^* = \argmin_{\mathbf{x} \in \mathbb{R}^d } f(\mathbf{x})
\end{equation*}
is one of the building blocks that many machine learning algorithms are based on. The difficulty of this problem varies significantly depending on the properties of $f$ and the way we can access information about it. The general case of non-convex functions makes the problem significantly more challenging, since first order stationary points can be global or local optima as well as saddle points. In fact, discovering global optima is an NP hard problem in general and even for quartic functions verifying local optima is a co-NP complete problem \citep{murty1987some,lee2017first}.

While local optima may be satisfactory for some applications in machine learning \cite{ChoromanskaHMAL15}, saddle points can make high dimensional non convex optimization tasks significantly more difficult \cite{DauphinPGCGB14,sun2018geometric}. Therefore, researchers have focused their efforts on functions possessing the strict saddle property. Under this property, Hessians of $f$ evaluated at saddle points have at least one negative eigenvalue making detection of saddle points tractable. Given this assumption, methods that use second order information like computing Hessians or Hessian-vector products \citep{nesterov2006cubic,carmon2016gradient,agarwal2017finding} can converge to second order stationary points (SOSPs) and thus avoid strict saddle points. Recent work \citep{ge2015escaping,levy2016power,jin2017escape,lee2017first,allen2018neon2,jin2018} has also showed that gradient descent (and its variants) can also avoid strict saddle points and converge to local minima.

Unfortunately access to gradient evaluations is not available in all settings of interest. Even with the advent of automatic differentiation software, there are several applications where computation of gradients is either computationally inefficient or even impossible. Examples of such applications are hyper-parameter tuning of machine learning models \cite{SnoekLA12,SalimansHCS17,ChoromanskiRSTW18}, black-box adversarial attacks on deep neural networks \cite{PapernotMGJCS17,MadryMSTV18,ChenZSYH17}, computer network control \cite{CCH18}, variational approaches to graphical models \cite{wainwright2008graphical} and simulation based \cite{rubinstein2016simulation,spall2005introduction} or bandit feedback optimization \cite{agarwal2010optimal,ChenG19}. Zero order methods, also known as black-box methods, try to address these issues by employing only evaluations of the function $f$ during the optimization procedure. The case of convex functions is well understood \cite{nesterov2017random,duchi2015optimal,agarwal2010optimal}. For the non-convex case, there has been a considerable amount of work on the convergence to first order stationary points both for deterministic settings \cite{nesterov2017random} and stochastic ones \cite{ghadimi2013stochastic,WangDBS18,Balasubramanian18,KCTCA18,gu2016zeroth}.

The case of \SOSPs has been so far comparatively under-studied. It has been established that \SOSPs are achievable through zero order trust region methods that employ fully quadratic models \cite{conn2009global}. The disadvantage of trust region methods is that their computation cost per iteration is $\mathcal{O}(d^4)$  which becomes quickly prohibitive as we increase the number of dimensions $d$. More recently, the authors of \cite{JinL0J18} studied the case of finding local minima of functions having access only to approximate function or gradient evaluations. They manage to reduce zero order optimization to the stochastic first order optimization of a Gaussian smoothed version of $f$. While this approach yields guarantees of convergence to \SOSPs, each stochastic gradient evaluation requires $\mathcal{O}(\operatorname{poly}(d, 1/\epsilon))$ number of function evaluations. This leads to significantly less efficient optimization algorithms when compared to their first order counterparts. \textbf{It is therefore yet unclear if there are scalable zero order methods that can safely avoid strict saddle points and always converge to local minima of $f$. To the best of our knowledge, our work is the first one to establish a positive answer to this important question.}

{\bf Our results.} \textit{We prove that zero order optimization methods solve general non-convex problems efficiently.} In a nutshell, we present a family of of zero order optimization methods which  provably converge to \SOSPs. Our proof includes a
new, elaborating analysis of Stable Manifold Theorem (See Section \ref{section:approx-grad-proof}). Additionally, the number of the approximate gradient evaluations match the standard bounds for first order methods in non-convex problems (see Table \ref{tab:main} \& Section \ref{section:efficient-approx-grad}).
\begin{table*}[h!]
  \begin{center}
    {
    \renewcommand{\arraystretch}{0.8}
    \begin{tabular}{llll}
       \toprule
    \textbf{Algorithm} & \textbf{Oracle} & \textbf{Iterations}  & \textbf{Evaluations of $f$} \\ 
      \midrule
      \textbf{Theorem} \ref{thm:second-order} & Approximate Gradient & Asymptotic  & Asymptotic\\
      \cite{lee2017first} & Exact Gradient  & Asymptotic & -\\
      \midrule
      \textbf{Theorem} \ref{thm:noise} & Approx. Gradient + Noise   & $\tilde{\mathcal{O}}(1/\epsilon^2)$ & $\tilde{\mathcal{O}}(d/\epsilon^2)$\\
      FPSGD \cite{JinL0J18} & Approx. Gradient + Noise & $\tilde{\mathcal{O}}(d /\epsilon^2)$ &  $\tilde{\mathcal{O}}(d^4/\epsilon^4)$\\
      ZPSGD \cite{JinL0J18} & Function Evaluations + Noise & $\tilde{\mathcal{O}}(1/\epsilon^2)$ &  $\tilde{\mathcal{O}}(d^2/\epsilon^5)$\\
      \cite{jin2017escape} & Exact Gradient + Noise  & $\tilde{\mathcal{O}}(1/\epsilon^2)$ & -\\
    \bottomrule
    \end{tabular}
      \caption{Oracle model and iteration complexity to \SOSPs. }
      \label{tab:main}
    }
  \end{center}
\end{table*}

{\bf Algorithms.} Instead of focusing on a single finite differences algorithms, we construct a general framework of approximate gradient oracles that generalizes over many finite differences approaches in the literature. We then use these approximate gradient oracles to devise approximate gradient descent algorithms. For more details see Section \ref{sec:grad-approx} and Definition \ref{def:well-behaved}.
    
{\bf Asymptotic convergence.}
We use the stable manifold theorem to prove that zero order methods can almost surely avoid saddle points. In contrast to the analysis of \cite{lee2017first} for first order methods, the zero order case is more demanding. Convergence to first order stationary points requires changing the gradient approximation accuracy over the iterations and, thus, the equivalent dynamical system is time dependent. By reducing our time dependent dynamical system to a time invariant one defined in an expanded state, we are able to obtain provable guarantees about avoiding saddle points. To extend our guarantees of convergence to deterministic choices of the initial accuracy, we provide a carefully tailored application of the Stable Manifold Theorem that analyzes the structure of the stable manifolds of the dynamical system. Our results on saddle point avoidance extend to functions with non isolated critical points. To address this, we provide sufficient conditions for point-wise convergence of the iterates of approximate gradient descent methods for the case of analytic functions. 

{\bf Convergence rates for noisy dynamics.}
In order to produce fast convergence rates, as in the case of first order methods \cite{jin2017escape}, it is useful to consider perturbed/noisy versions of the dynamics. Once again the case of zero order methods poses distinct hurdles. Close to critical points of $f$, approximations of the potentially arbitrarily small gradient can be very noisy. Iterates of exact gradient descent and approximate gradient descent may diverge significantly in this case. In fact, provably escaping saddle points by guaranteeing decrease of value of $f$ is more challenging for the case of approximate gradient descent since it is not a descent algorithm. A key technical step is to show that the negative curvature dynamics that enable gradient descent to escape saddle points are robust to gradient approximation errors. As long as the gradient approximation error is smaller than a fixed a-priori known threshold, zero order methods can provably escape saddle points. Based on this, we are able to prove that zero order methods can converge to approximate \SOSPs with the same number of approximate gradient evaluations provided by \cite{jin2017escape} up to constants.

It is worth pointing out that achieving an $\tilde{\mathcal{O}}(\epsilon^{-2})$ bound of approximate gradient evaluations requires conceptually different techniques from other recent approaches in zero order methods. Indeed, previous work on randomized and stochastic zero order optimization \citep{nesterov2017random,ghadimi2013stochastic} has relied on treating randomized approximate gradients of $f$ as in expectation exact gradients of a carefully constructed smoothed version of $f$. Then with some additional work, convergence arguments for the smooth version of $f$ can be transferred to $f$ itself. Although these arguments are applicable to our case as well, as shown by  the work of \cite{JinL0J18}, they also lead to a slowdown both in terms of the dimension $d$ and the required accuracy $\epsilon$. The main reasons behind this slowdown are that the Lipschitz constants of the smoothed version of $f$ depend on $d$ and the high variance of the stochastic gradient estimators. To sidestep both issues, we analyze the effect of gradient approximation error directly on the optimization of $f$.
 \section{Related Work}
Our work builds and improves upon previous finite difference approaches for non-convex optimization and provides \SOSP guarantees previously only reserved to computationally expensive methods. 

\textbf{First Order Algorithms} 
A recent line of work has shown that gradient descent and variations of it can actually converge to \SOSPs. Specifically, \cite{lee2017first} shows that gradient descent starting from a random point can eventually converge to \SOSPs with probability one. 
\citep{jin2017escape,jin2018} modified standard gradient descent using perturbations to provide an algorithm that converges to \SOSPs  in  $\mathcal{O}(\operatorname{poly} (\log d, 1/\epsilon))$ iterations. As noted in the introduction, the zero order case poses additional hurdles compared to the first order one. Our work, by addressing these hurdles effectively extends the guarantees provided by  \cite{lee2017first,jin2017escape} to zero order methods.

\textbf{Zero Order Algorithms} Approximating gradients using finite differences methods has been the standard approach for both for convex and non-convex zero order optimization.\cite{nesterov2017random} established convergence properties even for randomized gradient oracles. Recently, \citep{duchi2015optimal} provided optimal guarantees for stochastic convex optimization up to logarithmic factors. For the more general case of stochastic non-convex optimization there has been extensive work covering several aspects of the problem: distributed  \cite{HajinezhadZ18}, asynchronous  \cite{LianZHHL16}, high-dimensional \cite{WangDBS18,Balasubramanian18} optimization and  variance reduction \cite{KCTCA18,gu2016zeroth}. It is significant to mention that the aforementioned work is focused on convergence to $\epsilon-$first order stationary points.
 
Regarding \SOSPs, \cite{conn2009global} showed that trust region methods that employ fully quadratic models  can converge to \SOSPs at the cost of $\mathcal{O}(d^4)$ operations per iteration. The authors of \cite{JinL0J18} studied the convergence to \SOSPs using approximate function or gradient evaluations. While both approaches are applicable for the zero order setting with exact function evaluations, as we will see in Section \ref{sec:black-box}, this type of reduction results in algorithms that require substantially more function evaluations to reach an $\epsilon$-\SOSP. Our work provides provable guarantees of convergence at significantly faster rates.
 
\section{Preliminaries}
    \subsection{Notation}
    We will use lower case bold letters $\mathbf{x}, \mathbf{y}$ to denote vectors. $\norm{\cdot}$ will be used to denote the spectral norm and the $\ell_2$ vector norm. $\lambda_{min}(\cdot)$ will be used to denote the minimum eigenvalue of a matrix. If $g$ is a vector valued differentiable function then $D g$ denotes the differential of function $g$. We will use $\{e_1, e_2, \dots e_d\}$ to refer to the standard orthonormal basis of $\mathbb{R}^d$. Also $C^n$ is the set of $n$ times continuously differentiable functions. $B_\mathbf{x}(r)$ refers to the ball of radius $r$ centered at $\mathbf{x}$. Finally, $\mu(S)$ is the Lebesgue measure of a measurable set $S\subseteq \mathbb{R}^d$.

\subsection{Definitions}
    A function $f : \mathbb{R}^d \to \mathbb{R}$  is said to be $L$-continuous, $\ell$-gradient, $\rho$-Hessian Lipschitz if for every $\mathbf{x}, \mathbf{y} \in \mathbb{R}^d$ $ \norm{f (\mathbf{x}) -  f (\mathbf{y}) } \leq L \norm{\mathbf{x}-\mathbf{y}}$,
    $ \norm{\nabla f (\mathbf{x}) - \nabla f (\mathbf{y}) } \leq \ell \norm{\mathbf{x}-\mathbf{y}}$,
    $ \norm{\nabla^2 f (\mathbf{x}) - \nabla^2 f (\mathbf{y}) } \leq \rho \norm{\mathbf{x}-\mathbf{y}} $ correspondingly.  Additionally, we can define approximate first order stationary points as:
    \begin{definition}[$\epsilon$-first order stationary point]
     Let $f: \mathbb{R}^d \to \mathbb{R}$ be a differentiable function. Then $\mathbf{x} \in \mathbb{R}^d$ is a first order stationary point of $f$ if $\norm{\nabla f (\mathbf{x})} \leq \epsilon$.
    \end{definition}
    A first order stationary point can be either a local minimum, a local maximum or a saddle point. Following the terminology of \cite{lee2017first} and \cite{jin2017escape}, we will include local maxima in saddle points since they are both undesirable for our minimization task. Under this definition, strict saddle points can be identified as follows:
    \begin{definition}[Strict saddle point]
     Let $f: \mathbb{R}^d \to \mathbb{R}$ be a twice differentiable function. Then $\mathbf{x} \in \mathbb{R}^d$ is a strict saddle point of $f$ if $\norm{\nabla f (\mathbf{x})} = 0$ and $\lambda_{min}(\nabla^2 f (\mathbf{x})) < 0$.
    \end{definition}
    To avoid convergence to strict saddle points, we need to converge to \SOSPs. In order to study the convergence rate of algorithms that converge to \SOSPs, we need to define some notion of approximate \SOSPs. Following the convention of \cite{jin2017escape} we define the following:
    \begin{definition}[$\epsilon$-\SOSP]
     Let $f: \mathbb{R}^d \to \mathbb{R}$ be a $\rho$-Hessian Lipschitz  function. Then $\mathbf{x} \in \mathbb{R}^d$ is an $\epsilon$-second order order stationary point of $f$ if $\norm{\nabla f (\mathbf{x})} \leq \epsilon$ and $\lambda_{min}(\nabla^2 f (\mathbf{x})) \geq -\sqrt{\rho\epsilon}$.
    \end{definition}
    \subsection{Gradient Approximation using Zero Order Information} \label{sec:grad-approx}
    One of the key ways that enables zero order methods to converge quickly is using approximations of the gradient based on finite differences approaches. Here we will show how forward differencing can provide these approximate gradient calculations. Without much additional effort we can get the same results for other finite differences approaches like backward and symmetric difference as well as finite differences approaches with higher order accuracy guarantees. Let us define the gradient approximation function based on forward difference $r_f: \mathbb{R}^d \times \mathbb{R} \to \mathbb{R}^d$
    \begin{align}
        \label{algo:forward}
        r_{f}(\mathbf{x}, h) =  
        \begin{cases} 
        \sum_{l=0}^d \dfrac{f( \mathbf{x}+h \mathbf{e}_l)-f(\mathbf{x})}{h}\mathbf{e}_l \text{ when } h \neq 0 \\
        \nabla f (\mathbf{x}) \text{ if } h = 0
        \end{cases} 
    \end{align}
    This function takes two arguments: A vector $\mathbf{x}$ where the gradient should be approximated as well as a scalar value $h$ that controls the approximation accuracy of the estimator. An additional property that will be of interest when we analyze approximate gradient descent is the fact that $r_f$ is Lipschitz.  Based on the definition one can show: 
    \begin{lemma} \label{lemma:lipschitz-oracle} \label{lemma:bounded-error}
        Let $f$ be $\ell$-gradient Lipschitz. Then $r_{f}(\cdot, h)$ as defined in Equation \ref{algo:forward} is $\sqrt{d}\ell$ Lipschitz for all $h \in \mathbb{R}$ and $\forall h\in \mathbb{R}, \mathbf{x} \in \mathbb{R}^d :\norm{r_{f}(\mathbf{x},h) - \nabla f (\mathbf{x})} \leq  \ell\sqrt{d} \abs{h}.$
    \end{lemma}
     
    \subsection{Black box reductions to first order methods} \label{sec:black-box}
    As shown in the works of \cite{nesterov2017random,ghadimi2013stochastic}, zero order optimization is reducible to stochastic first order optimization. The reduction relies on treating randomized approximate gradients of $f$ as in expectation exact gradients of a carefully constructed smoothed version of $f$. These arguments are also applicable to our case as well. 
    FPSG, one of the approaches of \cite{JinL0J18}, naively leads to a large $\operatorname{poly}(d)$ dependence in the convergence rate. More specifically one can show that  \cite{JinL0J18}'s FPSG method needs $\tilde{\mathcal{O}}(d^3/\epsilon^4)$ evaluations of $\nabla g$ to converge to an $\epsilon$-\SOSP.  The main reason behind this dimension dependent slowdown is that the Hessian Lipschitz constant of the smoothed version of $g$ is $O(\rho \sqrt{d})$. An alternative approach in \cite{JinL0J18} named ZPSG builds gradient estimators using function evaluations directly.  The main source slowdown here is the high variance of the stohastic gradients. An analysis of those methods for the case where exact function evaluations are available can be found in the Appendix.
    
    In the next sections we will provide an alternative analysis that accounts for the gradient approximation errors on the optimization of $f$ directly. Thus, we will be able to sidestep the above issues and provide faster convergence rates and better sample complexity. \section{Approximate Gradient Descent}\label{section:approx-grad-proof}

\subsection{Description}
It is easy to see that conceptually any iterative optimization method can be expressed as a dynamical system of the form $\{  \mathbf{x}_{k+1} = g(\mathbf{x}_k) \}$   where $\mathbf{x}_k$ is the current solution iterate that gets updated through an update function $g$. Additionally, for first order methods  strict saddle points correspond to the unstable fixed points of the dynamical system. These key observations have motivated \cite{lee2017first} to use the Stable Manifold Theorem (SMT) \cite{shub1987global} in order to prove that gradient descent avoids strict saddle points. Intuitively, SMT formalizes why convergence to unstable fixed points is unlikely starting  from a local region around an unstable fixed point.  Adding the requirement that $g$ is a global diffeomorphism, \cite{lee2017first} generalizes the conclusions of SMT to the whole space. 
    
In order to prove similar guarantees for a zero order algorithm using approximate gradient evaluations, we will need to construct a new dynamical system that is applicable to our zero order setting. The state of our dynamical system $\boldsymbol{\chi}_k$ consists of two parts: The current solution iterate $\mathbf{x}_k$ that is a vector in $\mathbb{R}^d$ and a scalar value $h \in \mathbb{R}$ that controls the quality of the gradient approximation. Specifically we have
\begin{equation}
    \boldsymbol{\chi}_{k+1}=g_0(\boldsymbol{\chi}_k) \triangleq \binom{\mathbf{x}_{k+1}}{h_{k+1}} = \binom{\mathbf{x}_k - \eta q_x(\mathbf{x}_k,h_k)}{\beta q_h(h_k)} \label{eq:general-method}  
\end{equation}
where $\eta,\beta \in \mathbb{R}^+$ positive scalar parameters and functions $q_x: \mathbb{R}^d \times \mathbb{R} \to \mathbb{R}^d$ and $q_h: \mathbb{R} \to \mathbb{R}$. The function $q_x$ can be seen as the gradient approximation oracle used by the dynamical system as described in Section \ref{sec:grad-approx}. The function $q_h$ is responsible for controlling the accuracy of the gradient approximation. As we shall see later, it is important that $h_k$ converges to 0 so that the stable points of $g_0$ are the same as in gradient descent.
            
\subsection{Avoiding Strict Saddle points} \label{sec:avoid-saddles}
In this section we will provide sufficient conditions that the parameters $\eta,\beta$ must satisfy so that the update rule of Equation \ref{eq:general-method} avoids convergence to strict saddle points. To do this we will need to introduce some properties of $g_0$.
\begin{definition}[$(L,B,c)$-Well-behaved function] \label{def:well-behaved}
    Let $f: \mathbb{R}^d \to \mathbb{R} \in C^2$ be a $\ell$-gradient Lipschitz function. A function $g_0$ of the form of Equation \ref{eq:general-method} is a $(L,B,c)$-well behaved function (for function $f$) if it has the following properties:
    \begin{inparaenum}[i)]
                \item $q_x, q_h \in C^1$ with $q_h(0) = 0$.
                \item $\forall h \in \mathbb{R}: q_x(\cdot, h)$ is $L$ Lipschitz and  $0 < \frac{\partial q_h(h)}{\partial h} \leq B$.
                \item $\forall (\mathbf{x},h) \in \mathbb{R}^{d+1} : \,  \norm{q_x(\mathbf{x},h)- \nabla f(\mathbf{x})} \le c \abs{h}$.
    \end{inparaenum}
\end{definition}
Given this definition and Lemma \ref{lemma:lipschitz-oracle}, it is clear that we can always construct $(L,B,c)$-well-behaved functions for $L=\sqrt{d}\ell$, $B=1$, $c=\sqrt{d}\ell$ using $q_x = r_f$ and $q_h = h$.

In the following lemmas and theorems we will  require that $\beta B < 1$. Under this assumption $\beta q_h$ is a contraction having 0 as its only fixed point so for all fixed points of $g_0$ we know that $h=0$. Notice also that when $h=0$, we have $q_x(\mathbf{x},0) = \nabla f(\mathbf{x})$ and therefore the $\mathbf{x}$ coordinates of fixed points of $g_0$ must coincide with first order stationary points of $f$. In fact, in the Appendix we prove that there is a one to one mapping between strict saddles of $f$ and unstable fixed points of $g_0$. Using the same assumptions, we also get that  $\det(\mathrm{D}g_0 (\cdot)) \neq 0$. Putting all together, we are able to prove our first main result.	       
\begin{theorem}\label{thm:no-strict-exp}
Let $g_0$ be a $(L,B,c)$-well-behaved function for function $f$. Let $X_f^*$ be the set of strict saddle points of $f$. Then if $\eta< \frac{1}{L}$ and $\beta < \frac{1}{B}$: $\forall h_0\in \mathbb{R}: \mu(\{ \mathbf{x}_0: \lim_{k \to \infty} \mathbf{x}_k \in X_f^* \}) = 0$.
\end{theorem}
Notice that the random initialization refers only to the $\mathbf{x}_0$'s domain. Indeed a straightforward application of the result of \cite{lee2017first} would guarantee a saddle-avoidance lemma only under an extra random choice of $h_0$. Such a result would not be able to clarify if saddle-avoidance stems from the instability of the fixed point, just like in first order methods, or from the additional randomness of $h_0$. The key insight provided by the SMT is that the all the initialization points that eventually converge to an unstable fixed point lie in a low dimensional manifold. Thus, to obtain a stronger result we have to understand how SMT restricts the dimensionality of this stable manifold for a fixed $h_0$. The structure of the eigenvectors of the Jacobian of $g_0$ around a fixed point reveals that such an interesting decoupling is finally achievable. 

\subsection{Convergence} \label{sec:convergence}
In the previous section we provided sufficient conditions to avoid convergence to strict saddle points. These results are meaningful however only if $\lim_{k\to \infty} \mathbf{x}_k$ exists. Therefore, in this section we will provide sufficient conditions such that the dynamic system of $g_0$ converges. Given that strict saddle points are avoided, it is sufficient to prove convergence to first order stationary points. Let the error of the gradient approximation be 
$
    \boldsymbol{\varepsilon}_k = q_x(\mathbf{x}_k,h_k) - \nabla f(\mathbf{x}_k).
$
Firstly we establish the zero order analogue of the folklore lower bound for the decrease of the function:
\begin{lemma}[Step-Convergence]\label{lemma:step-convergence}
Suppose that $g_0$ is a $(L,B,c)$-well-behaved function for a $\ell$-gradient Lipschitz function $f$. If $\eta \leq \frac{1}{\ell}$ then we have that
$
    f(\mathbf{x}_{k+1}) \le f(\mathbf{x}_{k}) -\frac{\eta}{2}\left(\norm{\nabla f(\mathbf{x}_k)}^2-\norm{\boldsymbol{\varepsilon}_k}^2 \right).
$
\end{lemma}
Given this lemma we can prove convergence to first order stationary points.
\begin{theorem}[Convergence to first order stationary points] \label{theorem:exact-stationary-points}
Suppose that $g_0$ is a $(L,B,c)$-well-behaved function for a $\ell$-gradient Lipschitz function $f$. Let $\eta \leq \frac{1}{\ell}$, $\beta < \frac{1}{B}$. Then if $f$ is lower bounded~$\lim_{k \to \infty} \norm{\nabla f(\mathbf{x}_k)} = 0. $ 
\end{theorem}
The last theorem gives us a guarantee that the norm of the gradient is converging to zero but this is not enough to prove convergence to a single stationary point if $f$ has non isolated critical points. In the Appendix, we prove that if the gradient approximation error decreases quickly enough then convergence to a single stationary point is guaranteed for analytic functions.  This allows us to conclude our analysis with this final theorem.
            
\begin{theorem}[Convergence to minimizers]\label{thm:second-order}
Let $f: \mathbb{R}^d \to \mathbb{R} \in C^2$ be a $\ell$-gradient Lipschitz function. Let us also assume that $f$ is analytic, has compact sub-level sets and all of its saddle points are strict. Let $g_0$ be a $(L,B,c)$-well-behaved function for $f$ with $\eta < \min\{\frac{1}{L}, \frac{1}{2\ell} \}$ and $\beta < \frac{1-2\eta \ell}{B}$. If we pick a random initialization point $\mathbf{x}_0$, then we have that for the $\mathbf{x}_k$ iterates of $g_0$
\begin{equation*}
               \forall h_0 \in \mathbb{R} : \quad \Pr(\lim_{k \to \infty} \mathbf{x}_k = \mathbf{x}^*) = 1 
\end{equation*}
where $\mathbf{x}^*$ is a local minimizer of $f$.
\end{theorem}
 \section{Escaping Saddle Points Efficiently}\label{section:efficient-approx-grad}
     \subsection{Overview}
        In the previous subsections we provided sufficient conditions for approximate gradient descent to avoid strict saddle points. However, the stable manifold theorem guarantees that this will happen asymptotically. In fact, convergence could be quite slow until we reach a neighborhood of a local minimum. An analysis done for the first order case by \cite{du2017gradient} showed that avoiding saddle points could take exponential time in the worst case. In this section, we will use ideas from the work of \cite{jin2017escape} in order to get a zero order algorithm that converges to \SOSPs efficiently. 
        
        Convergence to \SOSPs poses unique challenges to zero order methods when it comes to controlling the gradient approximation accuracy. For convergence to first order stationary points one can use property iii) of Definition \ref{def:well-behaved} and Lemma \ref{lemma:step-convergence} to show that $h= \epsilon/c$ guarantees the decrease of $f$ until $\norm{\nabla  f(\mathbf{x}_k)}\leq \epsilon$. For \SOSPs, this is not applicable as the norm of the gradient can become arbitrarily small near saddle points. One could resort to iteratively trying smaller $h$ to find one that guarantees the decrease of $f$. A surprising fact about our algorithm is that even if the gradient is arbitrarily small, computationally burdensome searches for $h$ can be totally avoided.  
    \subsection{Algorithm}

\vspace{-0.5cm}
\begin{algorithm}[H]
 \renewcommand{\thealgorithm}{}
    \centering
    \caption{Initialization: $(\ell, \rho, \epsilon, c, \delta, \Delta_f)$}\label{algo:init}    
    \begin{algorithmic}[1]
        \State $\chi \leftarrow 3\max\{\log(\frac{d\ell\Delta_f}{c\epsilon^2\delta}), 4\}, 
            ~\eta \leftarrow \frac{c}{\ell},
            ~r \leftarrow \frac{\sqrt{c}}{\chi^2}\cdot\frac{\epsilon}{\ell}, 
            ~g_{\text{thres}} \leftarrow \frac{\sqrt{c}}{\chi^2}\cdot \epsilon,
            ~f_{\text{thres}} \leftarrow \frac{c}{\chi^3} \cdot \sqrt{\frac{\epsilon^3}{\rho}}$ 
        \State $~t_{\text{thres}} \leftarrow \frac{\chi}{c^2}\cdot\frac{\ell}{\sqrt{\rho \epsilon}},
            ~S \gets \frac{\sqrt{c}}{\chi}\frac{\sqrt{\rho \epsilon}}{\rho},
            ~h_{low} \leftarrow \frac{1}{c_h}\min\{g_{\text{thres}}, \frac{r\rho\delta S}{2\sqrt{d}} \}$
    \end{algorithmic}
\end{algorithm}
\vspace{-1cm}
\setcounter{algorithm}{0}
\begin{center}
\begin{minipage}{0.49\textwidth}
\begin{algorithm}[H]
    \centering
    \caption{\text{PAGD}($\mathbf{x}_0$)}\label{algo:PAGD}
    \begin{algorithmic}[1]
    \For{$t = 0, 1, \ldots $} \label{alg:marker}
            \State $\mathbf{z}_t \leftarrow q(\mathbf{x}_t, \frac{g_{\text{thres}}}{4c_h})$
            \If{$\norm{\mathbf{z}_t} \ge \frac{3}{4}g_{\text{thres}}$}\label{line:if} 
\State $\mathbf{x}_{t+1} \leftarrow \mathbf{x}_t - \eta \mathbf{z}_t$ 
            \Else \label{line:else} \State  $\mathbf{x}_{t+1} \leftarrow $ EscapeSaddle ($\mathbf{x}_t $)
                \IfOneLine{$\mathbf{x}_{t+1} = \mathbf{x}_{t}$}{\textbf{return } $\mathbf{x}_{t}$}
            \EndIf
            \EndFor
    \end{algorithmic}
\end{algorithm}
\end{minipage}
\hfill
\begin{minipage}{0.49\textwidth}
\begin{algorithm}[H]
    \centering
    \caption{EscapeSaddle ($\hat{\mathbf{x}}$)}\label{algo:EscapeSaddle}
    \begin{algorithmic}[1]
        \State $\boldsymbol{\xi} \sim\operatorname{Unif}(B_\mathbf{0}(r))$
        \State $\tilde{\mathbf{x}}_0 \leftarrow \hat{\mathbf{x}} + \boldsymbol{\xi}$
        \For{$i = 0, 1, \ldots t_{\text{thres}} $}
            \If{ $f(\hat{\mathbf{x}})  - f(\tilde{\mathbf{x}}_i)  \geq f_{\text{thres}}$}
                \State \textbf{return }  $\tilde{\mathbf{x}}_i$ 
            \EndIf
            \State $\tilde{\mathbf{x}}_{i+1} \leftarrow \tilde{\mathbf{x}}_i - \eta q( \tilde{\mathbf{x}}_i, h_{low})$
        \EndFor
        \State \textbf{return } $\hat{\mathbf{x}}$ 
    \end{algorithmic}
\end{algorithm}
\end{minipage}
\end{center}
     
    Just like \cite{jin2017escape}, we will assume that $f$ is $\ell-$gradient Lipschitz and also $\rho-$Hessian Lipschitz. To construct a zero order algorithm we will also need a gradient approximator $q:\mathbb{R}^d\times \mathbb{R}\to \mathbb{R}^d$. We will only require the error bound property on $q$, i.e., there exists a constant $c_h$ such that
    \begin{equation*}
    \forall \mathbf{x}\in \mathbb{R}^d, h \in \mathbb{R}:  \norm*{q(\mathbf{x},h) -\nabla f(\mathbf{x})} \le c_h\abs{h}     
    \end{equation*}
    The high level idea of Algorithm \ref{algo:PAGD} is that given a point $\mathbf{x}_t$ that is not an $\epsilon$-\SOSP the algorithm makes progress by finding a  $\mathbf{x}_{t+1}$ where $f(\mathbf{x}_{t+1})$  is substantially smaller than $f(\mathbf{x}_t)$. By the definition of $\epsilon$-\SOSPs either the gradient of $f$ at $\mathbf{x}_t$ is large or the Hessian has a substantially negative eigenvalue.
    
    Separating these two cases is not as straightforward as in the first order case. Given the norm of the approximate gradient $q(\mathbf{x},h)$, we only know that $\norm{\nabla f(\mathbf{x})} \in \norm{q(\mathbf{x},h)} \pm c_h \abs{h}$. In Algorithm~\ref{algo:PAGD} by choosing $3g_{\text{thres}}/4$ as the threshold to test for and $h = g_{\text{thres}}/(4c_h)$, we guarantee that in step~4 $\norm{\nabla f(\mathbf{x}_t)} \geq g_{\text{thres}}/2$. This threshold is actually high enough to guarantee substantial decrease of $f$. Indeed given that we have a lower bound on the exact gradient and using Lemma \ref{lemma:step-convergence} we get
    \begin{equation*}
        f(\mathbf{x}_{t})- f(\mathbf{x}_{t+1})\ge \frac{\eta}{2}\left(\norm{\nabla f(\mathbf{x}_t)}^2-\norm{\boldsymbol{\varepsilon}_t}^2 \right)\ge \tfrac{3}{32}\eta g_{\text{thres}}^2
    \end{equation*}
    where $\boldsymbol{\varepsilon}_t$ is the gradient approximation error at $\mathbf{x}_t$. This decrease is the same as in the first order case up to constants.
    
    On the other hand, in Algorithm \ref{algo:EscapeSaddle} we are guaranteed that $\norm{\nabla f(\mathbf{\hat{x}})} \leq g_{\text{thres}}$. In this case our approximate gradient cannot guarantee a substantial decrease of $f$. However, we know that the Hessian has a substantially negative eigenvalue and therefore a direction of steep decrease of $f$ must exist. The problem is that we do not know which direction has this property. In \cite{jin2017escape} it is proved that identifying this direction is not necessary for the first order case. Adding a small random perturbation to our current iterate (step 2) is enough so that with high probability we can get a substantial decrease of $f$ after at most $t_{\text{thres}}$ gradient descent steps (step 5). Of course this work is not directly applicable to our case since we do not have access to exact gradients.  
    
    The work of \cite{jin2017escape} mainly depends on two arguments to provide its guarantees. The first argument is that if the $\tilde{\mathbf{x}}_i$ iterates do not achieve a decrease of $f_\text{thres}$ in $t_{\text{thres}}$ steps then they must remain confined in a small ball around  $\tilde{\mathbf{x}}_0$. Specifically for the exact gradient case we have that
    \begin{equation*}
        \norm{\tilde{\mathbf{x}}_i - \tilde{\mathbf{x}}_0}^2 \leq 2\eta f_\text{thres} t_{\text{thres}}.
    \end{equation*}
    The zero order case is definitely more challenging since each update in Algorithm \ref{algo:EscapeSaddle} is not guaranteed to decrease the value of $f$. Therefore, iterates may wander away from $\tilde{\mathbf{x}}_0$ without even decreasing the function value of $f$. To amend this argument for the zero order case we require that $h_{low} \leq g_\text{thres}/c_h$. This guarantees that even if gradient approximation errors amass over the iterations we will get the same bound as the first order case up to constants.
    
    The second argument of \cite{jin2017escape} formalizes why the existence of a negative eigenvalue of the Hessian is important. Let us run gradient descent starting from two points $\mathbf{u}_0$ and $\mathbf{w}_0$ such that  $\mathbf{w}_0- \mathbf{u}_0 = \kappa \mathbf{e}$ where  $\mathbf{e}$ is the eigenvector corresponding to the most negative eigenvalue of the Hessian and  $\kappa \geq r\delta/(2\sqrt{d})$. Then at least one of the sequences $\{\mathbf{w}_i\},\{\mathbf{u}_i\}$ is able to escape away from its starting point in $t_{\text{thres}}$ iterations and by the first argument it is also able to decrease the value of $f$ substantially. The proof of the claim is based on creating a recurrence relationship on $\mathbf{v}_i = \mathbf{w}_i - \mathbf{u}_i$. The corresponding recurrence relationship for the zero order case is more complicated with additional terms that correspond to the gradient approximation errors for $\mathbf{w}_i$ and $\mathbf{u}_i$. However, we are able to prove that if $h_{low} \leq r\rho\delta S/(2\sqrt{d})$ then these additional terms cannot distort the exponential growth of $\mathbf{v}_i$. Having extended both arguments of \cite{jin2017escape} we can establish the same guarantees for escaping saddle points.
    \begin{theorem}[Analysis of PAGD]       \label{thm:noise}
        There exists absolute constant $c_{\max}$ such that: if $f$ is $\ell$-gradient Lipschitz and $\rho$-Hessian Lipschitz, then for any $\delta>0, \epsilon \le \frac{\ell^2}{\rho}, \Delta_f \ge f(\mathbf{x}_0) - f^\star$, and constant $c \le c_{\max}$, with probability $1-\delta$, the output of $\text{PAGD}(\mathbf{x}_0, \ell, \rho, \epsilon, c, \delta, \Delta_f)$ will be an $\epsilon$-\SOSP, and have the following number of iterations until termination:
        \begin{equation*}
            \mathcal{O}\left(\frac{\ell(f(\mathbf{x}_0) - f^\star)}{\epsilon^2}\log^{4}\left(\frac{d\ell\Delta_f}{\epsilon^2\delta}\right) \right)
        \end{equation*}
    \end{theorem}

 \section{Experiments}
In this section we use simulations to verify our theoretical findings. Specifically we are interested in verifying if zero order methods can avoid saddle points as efficiently as first order methods. To do this we use the two dimensional Rastrigin function, a popular benchmark in the non-convex optimization literature. This function exhibits several strict saddle points so it will be an adequate benchmark for our case. The two dimensional Rastrigin function  can be defined as
\begin{equation*}
    \textrm{Ras}(x_1, x_2) = 20 + x_1^2 -10\cos(2\pi x_1) + x_2^2 -10\cos(2\pi x_2).
\end{equation*}
 For this experiment we selected 75 points randomly from $[-1.5,1.5]\times[-1,5,1.5]$. In this domain the Rastrigin function is $\ell$-gradient Lipschitz with $\ell \approx 63.33$. Using these points as initialization we run gradient descent and the approximate gradient descent dynamical system we introduced in Section \ref{sec:avoid-saddles}. For both gradient descent and approximate gradient descent we used $\eta = 1/(4\ell)$. Then for approximate gradient descent we used symmetric differences to approximate the gradients and $\beta = 0.95$ as well as $h_0 = 0.15$.
\begin{figure}[h!]
 \centering
 \includegraphics[scale =0.3]{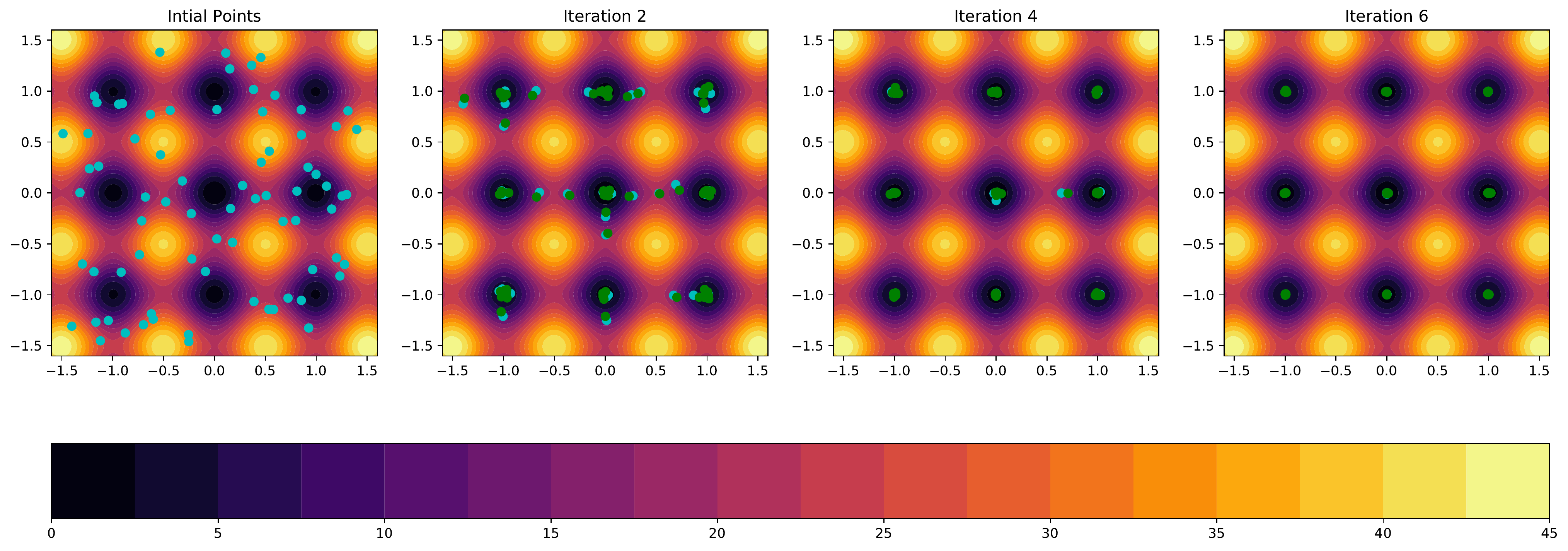}
 \vspace{-0.45cm}
 \caption{Contour plots of the Rastrigin function 
 along with the evolution of the iterates of gradient descent and approximate gradient descent. Green points correspond to gradient descent whereas cyan points correspond to approximate gradient descent.}
 \label{fig:rastrigin}
\end{figure}
Figure \ref{fig:rastrigin} shows the contour plot of the Rastrigin function as well as the evolution of the iterates of both methods. As expected, for points initialized closed to local minima of the function convergence is quite fast. On the other hand, points starting close to saddle points of the Rastrigin function take some more time to converge to minima. However, it is clear that in both cases the behaviour of gradient descent and approximate gradient descent is similar in the sense that for the same initialization there is no discrepancy in terms of convergence speed for the two methods. 

We also want to experimentally verify the performance of PAGD. To do this we use the octopus function proposed by \cite{du2017gradient}. This function is is particularly relevant to our setting as it possesses a sequence of saddle points. The authors of \cite{du2017gradient} proved that for this function gradient descent needs exponential time to avoid saddle points before converging to a local minimum. In contrast the perturbed version of gradient descent (PGD) of \cite{jin2017escape} does not suffer from the same limitation. Based on the results of Theorem \ref{thm:noise}, we expect PAGD to not have this limitation as well. We compare gradient descent (GD), PGD, AGD and PAGD on an octopus function of $d=15$ dimensions. Figure~\ref{figure:octupus} clearly shows that the zero order versions have the same iteration performance with the first-order ones. In fact, AGD is shown to behave even better than GD in this example thanks to the noise induced by the gradient approximation.

\begin{figure}[h!]
  \begin{center}
 \includegraphics[scale=0.38]{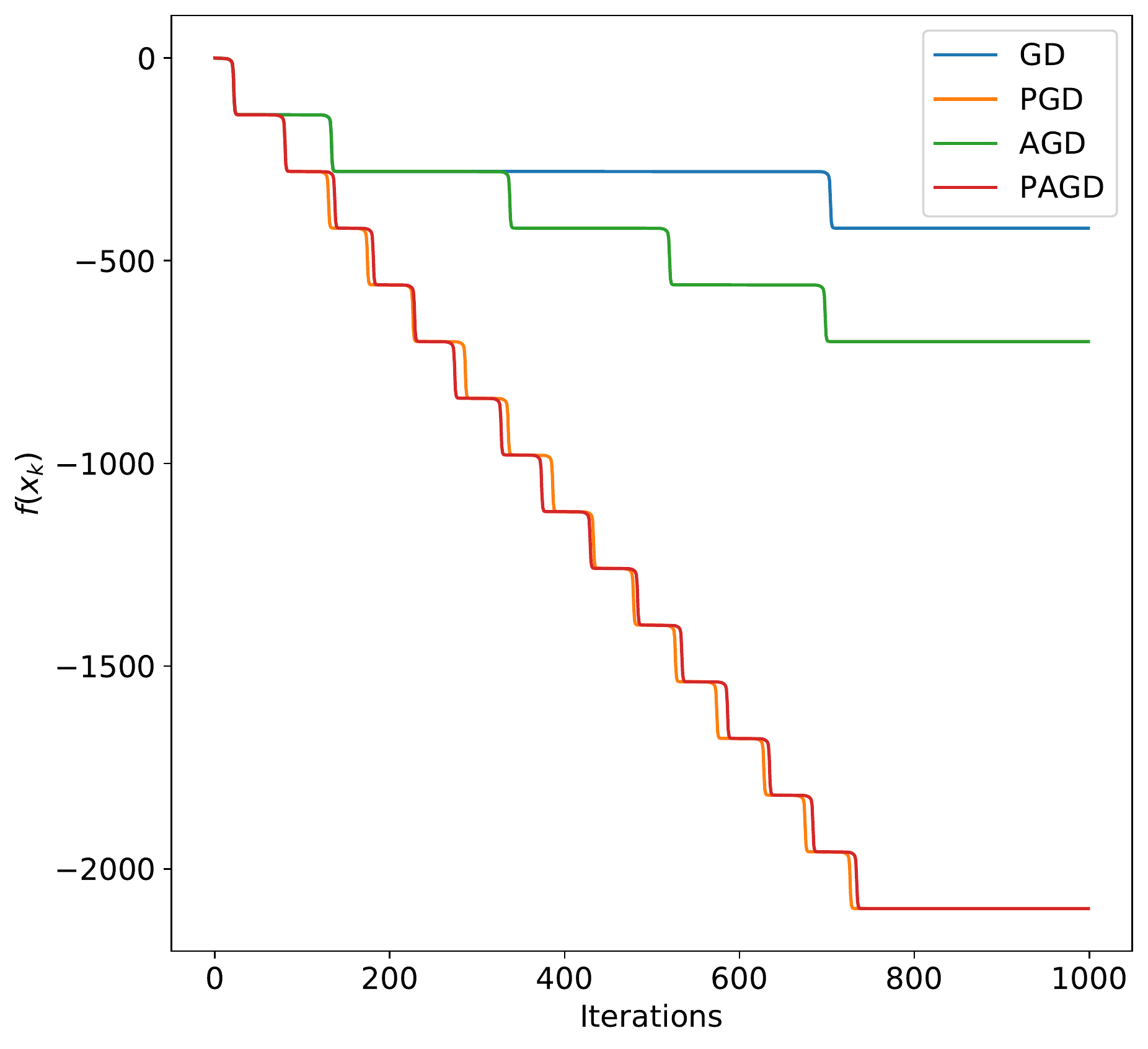}
  \end{center}
  \caption{Octopus function value varying the number of iterations. Parameters of the function $\tau = e$, $L = e$, $\gamma=1$. Parameters of first order methods taken from \cite{du2017gradient}. Zero order methods use symmetric differencing with $h=0.01$}
  \label{figure:octupus}
\end{figure}  \section{Conclusion}
This paper is the first one to establish that zero order methods can avoid saddle points efficiently. To achieve this we went beyond smoothing arguments used in prior work and studied the effect of the gradient approximation error on first order methods that converge to second order stationary points.  One important open question for future work is whether similar guarantees can be established for other zero order methods used in practice like direct search methods and trust region methods using linear models. Another generalization of interest would be to consider the performance of zero order methods for instances of (non-convex) constrained optimization.

\section*{Acknowledgements}\label{section:acknowledgements}
Georgios Piliouras acknowledges MOE AcRF Tier 2 Grant 2016-T2-1-170, grant PIE-SGP-AI-2018-01 and NRF 2018 Fellowship NRF-NRFF2018-07. Emmanouil-Vasileios Vlatakis-Gkaragkounis was supported by NSF CCF-1563155, NSF CCF-1814873,  NSF CCF-1703925, NSF CCF-1763970. We are grateful to Alexandros Potamianos  for bringing this problem to our attention, and for helpful discussions at an early stage of
this project for its connection to Natural Language Processing tasks. Finally, this work was supported by the Onassis Foundation - Scholarship ID: F ZN 010-1/2017-2018.

 \clearpage  

\iffalse
\appendix
\includepdf[pages=1-34]{appendix.pdf}
\fi

\bibliography{bibliography/references}

\begin{thebibliography}{40}
\providecommand{\natexlab}[1]{#1}
\providecommand{\url}[1]{\texttt{#1}}
\expandafter\ifx\csname urlstyle\endcsname\relax
  \providecommand{\doi}[1]{doi: #1}\else
  \providecommand{\doi}{doi: \begingroup \urlstyle{rm}\Url}\fi

\bibitem[Absil et~al.(2005)Absil, Mahony, and Andrews]{absil2005convergence}
Pierre{-}Antoine Absil, Robert~E. Mahony, and B.~Andrews.
\newblock Convergence of the iterates of descent methods for analytic cost
  functions.
\newblock \emph{{SIAM} Journal on Optimization}, 16\penalty0 (2):\penalty0
  531--547, 2005.
\newblock \doi{10.1137/040605266}.

\bibitem[Agarwal et~al.(2010)Agarwal, Dekel, and Xiao]{agarwal2010optimal}
Alekh Agarwal, Ofer Dekel, and Lin Xiao.
\newblock Optimal algorithms for online convex optimization with multi-point
  bandit feedback.
\newblock In \emph{{COLT} 2010 - The 23rd Conference on Learning Theory, Haifa,
  Israel, June 27-29, 2010}, pages 28--40, 2010.

\bibitem[Agarwal et~al.(2017)Agarwal, {Allen Zhu}, Bullins, Hazan, and
  Ma]{agarwal2017finding}
Naman Agarwal, Zeyuan {Allen Zhu}, Brian Bullins, Elad Hazan, and Tengyu Ma.
\newblock Finding approximate local minima faster than gradient descent.
\newblock In \emph{Proceedings of the 49th Annual {ACM} {SIGACT} Symposium on
  Theory of Computing, {STOC} 2017, Montreal, QC, Canada, June 19-23, 2017},
  pages 1195--1199, 2017.
\newblock \doi{10.1145/3055399.3055464}.

\bibitem[Allen{-}Zhu and Li(2018)]{allen2018neon2}
Zeyuan Allen{-}Zhu and Yuanzhi Li.
\newblock {NEON2:} finding local minima via first-order oracles.
\newblock In \emph{Advances in Neural Information Processing Systems 31: Annual
  Conference on Neural Information Processing Systems 2018, NeurIPS 2018, 3-8
  December 2018, Montr{\'{e}}al, Canada.}, pages 3720--3730, 2018.

\bibitem[Balasubramanian and Ghadimi(2018)]{Balasubramanian18}
Krishnakumar Balasubramanian and Saeed Ghadimi.
\newblock Zeroth-order (non)-convex stochastic optimization via conditional
  gradient and gradient updates.
\newblock In \emph{Advances in Neural Information Processing Systems 31: Annual
  Conference on Neural Information Processing Systems 2018, NeurIPS 2018, 3-8
  December 2018, Montr{\'{e}}al, Canada.}, pages 3459--3468, 2018.
\newblock URL
  \url{http://papers.nips.cc/paper/7605-zeroth-order-non-convex-stochastic-optimization-via-conditional-gradient-and-gradient-updates}.

\bibitem[Carmon and Duchi(2016)]{carmon2016gradient}
Yair Carmon and John~C. Duchi.
\newblock Gradient descent efficiently finds the cubic-regularized non-convex
  newton step.
\newblock \emph{CoRR}, abs/1612.00547, 2016.

\bibitem[Chen et~al.(2017)Chen, Zhang, Sharma, Yi, and Hsieh]{ChenZSYH17}
Pin{-}Yu Chen, Huan Zhang, Yash Sharma, Jinfeng Yi, and Cho{-}Jui Hsieh.
\newblock {ZOO:} zeroth order optimization based black-box attacks to deep
  neural networks without training substitute models.
\newblock In \emph{Proceedings of the 10th {ACM} Workshop on Artificial
  Intelligence and Security, AISec@CCS 2017, Dallas, TX, USA, November 3,
  2017}, pages 15--26, 2017.
\newblock \doi{10.1145/3128572.3140448}.
\newblock URL \url{https://doi.org/10.1145/3128572.3140448}.

\bibitem[Chen and Giannakis(2019)]{ChenG19}
Tianyi Chen and Georgios~B. Giannakis.
\newblock Bandit convex optimization for scalable and dynamic iot management.
\newblock \emph{{IEEE} Internet of Things Journal}, 6\penalty0 (1):\penalty0
  1276--1286, 2019.
\newblock \doi{10.1109/JIOT.2018.2839563}.
\newblock URL \url{https://doi.org/10.1109/JIOT.2018.2839563}.

\bibitem[Choromanska et~al.(2015)Choromanska, Henaff, Mathieu, Arous, and
  LeCun]{ChoromanskaHMAL15}
Anna Choromanska, Mikael Henaff, Micha{\"{e}}l Mathieu, G{\'{e}}rard~Ben Arous,
  and Yann LeCun.
\newblock The loss surfaces of multilayer networks.
\newblock In \emph{Proceedings of the Eighteenth International Conference on
  Artificial Intelligence and Statistics, {AISTATS} 2015, San Diego,
  California, USA, May 9-12, 2015}, 2015.
\newblock URL \url{http://jmlr.org/proceedings/papers/v38/choromanska15.html}.

\bibitem[Choromanski et~al.(2018)Choromanski, Rowland, Sindhwani, Turner, and
  Weller]{ChoromanskiRSTW18}
Krzysztof Choromanski, Mark Rowland, Vikas Sindhwani, Richard~E. Turner, and
  Adrian Weller.
\newblock Structured evolution with compact architectures for scalable policy
  optimization.
\newblock In \emph{Proceedings of the 35th International Conference on Machine
  Learning, {ICML} 2018, Stockholmsm{\"{a}}ssan, Stockholm, Sweden, July 10-15,
  2018}, pages 969--977, 2018.
\newblock URL \url{http://proceedings.mlr.press/v80/choromanski18a.html}.

\bibitem[Conn et~al.(2009)Conn, Scheinberg, and Vicente]{conn2009global}
Andrew~R. Conn, Katya Scheinberg, and Lu{\'{\i}}s~N. Vicente.
\newblock Global convergence of general derivative-free trust-region algorithms
  to first- and second-order critical points.
\newblock \emph{{SIAM} Journal on Optimization}, 20\penalty0 (1):\penalty0
  387--415, 2009.
\newblock \doi{10.1137/060673424}.

\bibitem[Dauphin et~al.(2014)Dauphin, Pascanu, G{\"{u}}l{\c{c}}ehre, Cho,
  Ganguli, and Bengio]{DauphinPGCGB14}
Yann~N. Dauphin, Razvan Pascanu, {\c{C}}aglar G{\"{u}}l{\c{c}}ehre, KyungHyun
  Cho, Surya Ganguli, and Yoshua Bengio.
\newblock Identifying and attacking the saddle point problem in
  high-dimensional non-convex optimization.
\newblock In \emph{Advances in Neural Information Processing Systems 27: Annual
  Conference on Neural Information Processing Systems 2014, December 8-13 2014,
  Montreal, Quebec, Canada}, pages 2933--2941, 2014.
\newblock URL
  \url{http://papers.nips.cc/paper/5486-identifying-and-attacking-the-saddle-point-problem-in-high-dimensional-non-convex-optimization}.

\bibitem[Du et~al.(2017)Du, Jin, Lee, Jordan, Singh, and
  P{\'{o}}czos]{du2017gradient}
Simon~S. Du, Chi Jin, Jason~D. Lee, Michael~I. Jordan, Aarti Singh, and
  Barnab{\'{a}}s P{\'{o}}czos.
\newblock Gradient descent can take exponential time to escape saddle points.
\newblock In \emph{Advances in Neural Information Processing Systems 30: Annual
  Conference on Neural Information Processing Systems 2017, 4-9 December 2017,
  Long Beach, CA, {USA}}, pages 1067--1077, 2017.

\bibitem[Duchi et~al.(2015)Duchi, Jordan, Wainwright, and
  Wibisono]{duchi2015optimal}
John~C. Duchi, Michael~I. Jordan, Martin~J. Wainwright, and Andre Wibisono.
\newblock Optimal rates for zero-order convex optimization: The power of two
  function evaluations.
\newblock \emph{{IEEE} Trans. Information Theory}, 61\penalty0 (5):\penalty0
  2788--2806, 2015.
\newblock \doi{10.1109/TIT.2015.2409256}.

\bibitem[Ge et~al.(2015)Ge, Huang, Jin, and Yuan]{ge2015escaping}
Rong Ge, Furong Huang, Chi Jin, and Yang Yuan.
\newblock Escaping from saddle points - online stochastic gradient for tensor
  decomposition.
\newblock In \emph{Proceedings of The 28th Conference on Learning Theory,
  {COLT} 2015, Paris, France, July 3-6, 2015}, pages 797--842, 2015.

\bibitem[Ghadimi and Lan(2013)]{ghadimi2013stochastic}
Saeed Ghadimi and Guanghui Lan.
\newblock Stochastic first- and zeroth-order methods for nonconvex stochastic
  programming.
\newblock \emph{{SIAM} Journal on Optimization}, 23\penalty0 (4):\penalty0
  2341--2368, 2013.
\newblock \doi{10.1137/120880811}.

\bibitem[Gu et~al.(2016)Gu, Huo, and Huang]{gu2016zeroth}
Bin Gu, Zhouyuan Huo, and Heng Huang.
\newblock Zeroth-order asynchronous doubly stochastic algorithm with variance
  reduction.
\newblock \emph{arXiv preprint arXiv:1612.01425}, 2016.

\bibitem[Hajinezhad and Zavlanos(2018)]{HajinezhadZ18}
Davood Hajinezhad and Michael~M. Zavlanos.
\newblock Gradient-free multi-agent nonconvex nonsmooth optimization.
\newblock In \emph{57th {IEEE} Conference on Decision and Control, {CDC} 2018,
  Miami, FL, USA, December 17-19, 2018}, pages 4939--4944, 2018.
\newblock \doi{10.1109/CDC.2018.8619333}.
\newblock URL \url{https://doi.org/10.1109/CDC.2018.8619333}.

\bibitem[Jin et~al.(2017)Jin, Ge, Netrapalli, Kakade, and
  Jordan]{jin2017escape}
Chi Jin, Rong Ge, Praneeth Netrapalli, Sham~M. Kakade, and Michael~I. Jordan.
\newblock How to escape saddle points efficiently.
\newblock In \emph{Proceedings of the 34th International Conference on Machine
  Learning, {ICML} 2017, Sydney, NSW, Australia, 6-11 August 2017}, pages
  1724--1732, 2017.

\bibitem[Jin et~al.(2018{\natexlab{a}})Jin, Liu, Ge, and Jordan]{JinL0J18}
Chi Jin, Lydia~T. Liu, Rong Ge, and Michael~I. Jordan.
\newblock On the local minima of the empirical risk.
\newblock In \emph{Advances in Neural Information Processing Systems 31: Annual
  Conference on Neural Information Processing Systems 2018, NeurIPS 2018, 3-8
  December 2018, Montr{\'{e}}al, Canada.}, pages 4901--4910,
  2018{\natexlab{a}}.
\newblock URL
  \url{http://papers.nips.cc/paper/7738-on-the-local-minima-of-the-empirical-risk}.

\bibitem[Jin et~al.(2018{\natexlab{b}})Jin, Netrapalli, and Jordan]{jin2018}
Chi Jin, Praneeth Netrapalli, and Michael~I. Jordan.
\newblock Accelerated gradient descent escapes saddle points faster than
  gradient descent.
\newblock In S\'ebastien Bubeck, Vianney Perchet, and Philippe Rigollet,
  editors, \emph{Proceedings of the 31st Conference On Learning Theory},
  volume~75 of \emph{Proceedings of Machine Learning Research}, pages
  1042--1085. PMLR, 06--09 Jul 2018{\natexlab{b}}.

\bibitem[Lee et~al.(2019)Lee, Panageas, Piliouras, Simchowitz, Jordan, and
  Recht]{lee2017first}
Jason~D. Lee, Ioannis Panageas, Georgios Piliouras, Max Simchowitz, Michael~I.
  Jordan, and Benjamin Recht.
\newblock First-order methods almost always avoid strict saddle points.
\newblock \emph{Math. Program.}, 176\penalty0 (1-2):\penalty0 311--337, 2019.
\newblock \doi{10.1007/s10107-019-01374-3}.
\newblock URL \url{https://doi.org/10.1007/s10107-019-01374-3}.

\bibitem[Levy(2016)]{levy2016power}
Kfir~Y. Levy.
\newblock The power of normalization: Faster evasion of saddle points.
\newblock \emph{CoRR}, abs/1611.04831, 2016.

\bibitem[Lian et~al.(2016)Lian, Zhang, Hsieh, Huang, and Liu]{LianZHHL16}
Xiangru Lian, Huan Zhang, Cho{-}Jui Hsieh, Yijun Huang, and Ji~Liu.
\newblock A comprehensive linear speedup analysis for asynchronous stochastic
  parallel optimization from zeroth-order to first-order.
\newblock In \emph{Advances in Neural Information Processing Systems 29: Annual
  Conference on Neural Information Processing Systems 2016, December 5-10,
  2016, Barcelona, Spain}, pages 3054--3062, 2016.
\newblock URL
  \url{http://papers.nips.cc/paper/6551-a-comprehensive-linear-speedup-analysis-for-asynchronous-stochastic-parallel-optimization-from-zeroth-order-to-first-order}.

\bibitem[Liu et~al.(2018{\natexlab{a}})Liu, Chen, Chen, and Hero]{CCH18}
Sijia Liu, Jie Chen, Pin{-}Yu Chen, and Alfred Hero.
\newblock Zeroth-order online alternating direction method of multipliers:
  Convergence analysis and applications.
\newblock In \emph{International Conference on Artificial Intelligence and
  Statistics, {AISTATS} 2018, 9-11 April 2018, Playa Blanca, Lanzarote, Canary
  Islands, Spain}, pages 288--297, 2018{\natexlab{a}}.
\newblock URL \url{http://proceedings.mlr.press/v84/liu18a.html}.

\bibitem[Liu et~al.(2018{\natexlab{b}})Liu, Kailkhura, Chen, Ting, Chang, and
  Amini]{KCTCA18}
Sijia Liu, Bhavya Kailkhura, Pin{-}Yu Chen, Pai{-}Shun Ting, Shiyu Chang, and
  Lisa Amini.
\newblock Zeroth-order stochastic variance reduction for nonconvex
  optimization.
\newblock In \emph{Advances in Neural Information Processing Systems 31: Annual
  Conference on Neural Information Processing Systems 2018, NeurIPS 2018, 3-8
  December 2018, Montr{\'{e}}al, Canada.}, pages 3731--3741,
  2018{\natexlab{b}}.
\newblock URL
  \url{http://papers.nips.cc/paper/7630-zeroth-order-stochastic-variance-reduction-for-nonconvex-optimization}.

\bibitem[Loring(2008)]{loring2008introduction}
W~Tu Loring.
\newblock An introduction to manifolds, 2008.

\bibitem[Madry et~al.(2018)Madry, Makelov, Schmidt, Tsipras, and
  Vladu]{MadryMSTV18}
Aleksander Madry, Aleksandar Makelov, Ludwig Schmidt, Dimitris Tsipras, and
  Adrian Vladu.
\newblock Towards deep learning models resistant to adversarial attacks.
\newblock In \emph{6th International Conference on Learning Representations,
  {ICLR} 2018, Vancouver, BC, Canada, April 30 - May 3, 2018, Conference Track
  Proceedings}, 2018.
\newblock URL \url{https://openreview.net/forum?id=rJzIBfZAb}.

\bibitem[Murty and Kabadi(1987)]{murty1987some}
Katta~G. Murty and Santosh~N. Kabadi.
\newblock Some np-complete problems in quadratic and nonlinear programming.
\newblock \emph{Math. Program.}, 39\penalty0 (2):\penalty0 117--129, 1987.
\newblock \doi{10.1007/BF02592948}.

\bibitem[Nesterov and Polyak(2006)]{nesterov2006cubic}
Yurii Nesterov and Boris~T. Polyak.
\newblock Cubic regularization of newton method and its global performance.
\newblock \emph{Math. Program.}, 108\penalty0 (1):\penalty0 177--205, 2006.
\newblock \doi{10.1007/s10107-006-0706-8}.

\bibitem[Nesterov and Spokoiny(2017)]{nesterov2017random}
Yurii Nesterov and Vladimir~G. Spokoiny.
\newblock Random gradient-free minimization of convex functions.
\newblock \emph{Foundations of Computational Mathematics}, 17\penalty0
  (2):\penalty0 527--566, 2017.
\newblock \doi{10.1007/s10208-015-9296-2}.

\bibitem[Papernot et~al.(2017)Papernot, McDaniel, Goodfellow, Jha, Celik, and
  Swami]{PapernotMGJCS17}
Nicolas Papernot, Patrick~D. McDaniel, Ian~J. Goodfellow, Somesh Jha, Z.~Berkay
  Celik, and Ananthram Swami.
\newblock Practical black-box attacks against machine learning.
\newblock In \emph{Proceedings of the 2017 {ACM} on Asia Conference on Computer
  and Communications Security, AsiaCCS 2017, Abu Dhabi, United Arab Emirates,
  April 2-6, 2017}, pages 506--519, 2017.
\newblock \doi{10.1145/3052973.3053009}.
\newblock URL \url{https://doi.org/10.1145/3052973.3053009}.

\bibitem[Rubinstein and Kroese(2016)]{rubinstein2016simulation}
Reuven~Y Rubinstein and Dirk~P Kroese.
\newblock \emph{Simulation and the Monte Carlo method}, volume~10.
\newblock John Wiley \& Sons, 2016.

\bibitem[Salimans et~al.(2017)Salimans, Ho, Chen, and Sutskever]{SalimansHCS17}
Tim Salimans, Jonathan Ho, Xi~Chen, and Ilya Sutskever.
\newblock Evolution strategies as a scalable alternative to reinforcement
  learning.
\newblock \emph{CoRR}, abs/1703.03864, 2017.
\newblock URL \url{http://arxiv.org/abs/1703.03864}.

\bibitem[Shub(1987)]{shub1987global}
Michael Shub.
\newblock \emph{Global stability of dynamical systems}.
\newblock Springer Science \& Business Media, 1987.

\bibitem[Snoek et~al.(2012)Snoek, Larochelle, and Adams]{SnoekLA12}
Jasper Snoek, Hugo Larochelle, and Ryan~P. Adams.
\newblock Practical bayesian optimization of machine learning algorithms.
\newblock In \emph{Advances in Neural Information Processing Systems 25: 26th
  Annual Conference on Neural Information Processing Systems 2012. Proceedings
  of a meeting held December 3-6, 2012, Lake Tahoe, Nevada, United States.},
  pages 2960--2968, 2012.
\newblock URL
  \url{http://papers.nips.cc/paper/4522-practical-bayesian-optimization-of-machine-learning-algorithms}.

\bibitem[Spall(2003)]{spall2005introduction}
James~C. Spall.
\newblock \emph{Introduction to Stochastic Search and Optimization}.
\newblock John Wiley \& Sons, Inc., New York, NY, USA, 1 edition, 2003.
\newblock ISBN 0471330523.

\bibitem[Sun et~al.(2018)Sun, Qu, and Wright]{sun2018geometric}
Ju~Sun, Qing Qu, and John Wright.
\newblock A geometric analysis of phase retrieval.
\newblock \emph{Foundations of Computational Mathematics}, 18\penalty0
  (5):\penalty0 1131--1198, 2018.
\newblock \doi{10.1007/s10208-017-9365-9}.

\bibitem[Wainwright and Jordan(2008)]{wainwright2008graphical}
Martin~J. Wainwright and Michael~I. Jordan.
\newblock Graphical models, exponential families, and variational inference.
\newblock \emph{Foundations and Trends in Machine Learning}, 1\penalty0
  (1-2):\penalty0 1--305, 2008.
\newblock \doi{10.1561/2200000001}.

\bibitem[Wang et~al.(2018)Wang, Du, Balakrishnan, and Singh]{WangDBS18}
Yining Wang, Simon~S. Du, Sivaraman Balakrishnan, and Aarti Singh.
\newblock Stochastic zeroth-order optimization in high dimensions.
\newblock In \emph{International Conference on Artificial Intelligence and
  Statistics, {AISTATS} 2018, 9-11 April 2018, Playa Blanca, Lanzarote, Canary
  Islands, Spain}, pages 1356--1365, 2018.
\newblock URL \url{http://proceedings.mlr.press/v84/wang18e.html}.

\end{thebibliography}
\bibliographystyle{plainnat}
\clearpage

\vbox{\hsize\textwidth
		\linewidth\hsize
		\vskip 0.1in
		\hrule height 4pt
		\vskip 0.25in
		\centering
		{\LARGE\bf Efficiently avoiding saddle points\\ with zero order methods: No gradients required\\ \vspace{.1in}\large Supplementary Materials}
		\vskip 0.29in
		\hrule height 1pt
		
	}
	
\setcounter{lemma}{3}
\setcounter{theorem}{4}
\setcounter{algorithm}{2}
\appendix

\section{Preliminaries Detailed proofs}
\shadowbox{
\begin{minipage}[c]{5in}
In this first subsection, we show that the forward finite differences method can be used to construct an approximate gradient oracle. Similar oracles can be constructed using backward, symmetric finite differences or Richardson extrapolation which have even higher gradient approximation accuracy. Additionally, we compute the Lipschitz constant of our method and we show that our definition of "well-behaved" approximate gradient is well defined. In other words, there are simple approximation oracles which follow the smoothness requirements that our work assumes.
\end{minipage}
}
\subsection{Gradient Approximation using Zero Order Information}
    \begin{lemma}[ Lemma \ref{lemma:lipschitz-oracle} restated ]
        Let $f$ be $\ell$-gradient Lipschitz. Then $r_{f}(\cdot, h)$ as defined in Equation \ref{algo:forward} is $\sqrt{d}\ell$ Lipschitz for all $h \in \mathbb{R}$ and it holds that:  $ \norm{r_{f}(\mathbf{x},h) - \nabla f (\mathbf{x})} \leq  \ell\sqrt{d} \abs{h}    $
    \end{lemma}
    \begin{proof}
    For the first part of the lemma we split our proof into two cases:
    \begin{itemize}
        \item For any $h \neq 0$ and any $\mathbf{x},\mathbf{x}' \in \mathbb{R}^d$ we have
        \begin{align*}
            \norm{r_{f}(\mathbf{x},h) - r_{f}(\mathbf{x}',h)} 
            &=\norm*{\sum_{l=0}^d \frac{f(\mathbf{x}+h\mathbf{e}_l)-f(\mathbf{x})}{h}\mathbf{e}_l
	        -\sum_{l=0}^d \frac{f(\mathbf{x}'+h\mathbf{e}_l)-f(\mathbf{x}')}{h}\mathbf{e}_l} \\
	    &=\sqrt{\sum_{l=0}^d \abs*{\frac{f(\mathbf{x}+h\mathbf{e}_l)-f(\mathbf{x}'+h\mathbf{e}_l)
	    -(f(\mathbf{x})-f(\mathbf{x}'))}{h}}^2}
        \end{align*}
        Let us define the function $q_l(s)= f(\mathbf{x}+s\mathbf{e}_l)-f(\mathbf{x}'+s\mathbf{e}_l)$ for all $l \in [d]$. Then by applying the mean value theorem we get
        \begin{align*}
           \norm{r_{f}(\mathbf{x},h) -r_{f}(\mathbf{x}',h)} 
            =\sqrt{\sum_{l=0}^d
	        \abs*{\frac{q_l(h)-q_l(0)}{h}}^2}=
            \sqrt{\sum_{l=0}^d
	        \abs{q_l'(\xi_l)}^2} 
        \end{align*}
        for some $\xi_l \in (0,h)$. We have that $q_l'(\xi_l)=\frac{\partial f(\mathbf{x}+\xi_l\mathbf{e}_l)}{\partial x_l}-\frac{\partial f(\mathbf{x}'+\xi_l\mathbf{e}_l)}{\partial x_l}$. If $f$ is $\ell$-gradient Lipschitz so are all the partial derivatives
        \begin{align*}
          \norm{r_{f}(\mathbf{x},h) -r_{f}(\mathbf{x}',h)} \leq 
          \sqrt{\sum_{l=0}^d \ell^2 \norm{\mathbf{x}-\mathbf{x}'}^2} 
          = \sqrt{d}\ell \norm{\mathbf{x}-\mathbf{x}'}
        \end{align*}
        \item For the special case of $h=0$
        \begin{align*}
          \norm{r_{f}(\mathbf{x},0) -r_{f}(\mathbf{x}',0)}
          &=\norm{\nabla f(\mathbf{x}) -\nabla f(\mathbf{x}') }
         \leq \ell \norm{\mathbf{x}-\mathbf{x}'} \leq \sqrt{d} \ell \norm{\mathbf{x}-\mathbf{x}'} 
        \end{align*}
        \end{itemize} Similarly, for the second part of the lemma we have that
        for any $h \neq 0$ and any $\mathbf{x}$ 
        \begin{align*}
         \norm{r_{f}(\mathbf{x},h) - \nabla f (\mathbf{x})} 
         &=\norm*{\sum_{l=0}^d \frac{f(\mathbf{x}+h\mathbf{e}_l)-f(\mathbf{x})}{h}\mathbf{e}_l - \nabla f(\mathbf{x})}\\
         &=\sqrt{\sum_{l=0}^d \abs*{\frac{f(\mathbf{x}+h\mathbf{e}_l)-f(\mathbf{x})}{h} - \frac{\partial f(\mathbf{x})}{\partial x_l}}^2}
        \end{align*}
        For each $l \in [d]$ we use the mean value theorem so that for some $x_l: \abs{\xi_l} \leq \abs{h}$ we have
        \begin{align*}
         \norm{r_{f}(\mathbf{x},h) - \nabla f (\mathbf{x})} 
         &=\sqrt{\sum_{l=0}^d \abs*{\frac{\partial f(\mathbf{x + \xi_l \mathbf{e}_l})}{\partial x_l} - \frac{\partial f(\mathbf{x})}{\partial x_l}}^2}\\
         &\leq  \sqrt{\sum_{l=0}^d (\ell\xi_l )^2} \leq \ell\sqrt{d}\abs{h}
        \end{align*}
        For $h=0$ the requested inequality holds as an equality.
        \end{proof}
\shadowbox{
\begin{minipage}[c]{5in}
As noted in the main paper, recent studies have analyzed zero order optimization by carefully crafting a smoothed version of the original objective function. These arguments are also applicable to our case as well.The following lemmas show why these approaches  lead $\operatorname{poly}(d,\epsilon^{-1})$ slowdown in terms of number of iterations and function evaluations.
\end{minipage}
}
\subsection{Black box reductions to first order methods}
Algorithm \ref{algo:FPSGD} of \cite{JinL0J18}, uses approximate gradient evaluations at randomly sampled points around the current iterate to get an estimate of the gradient of $f$. This estimate is then perturbed with noise in order to avoid any potential saddle point.
\begin{algorithm}[h!]
	        \caption{First order Perturbed Stochastic Gradient Descent (FPSGD)}\label{algo:FPSGD}
	        \begin{algorithmic}
		        \renewcommand{\algorithmicrequire}{\textbf{Input: }}
		        \renewcommand{\algorithmicensure}{\textbf{Output: }}
		        \Require $\mathbf{x}_0$, learning rate $\eta$, noise radius $r$, mini-batch size $m$.
		        \For{$t = 0, 1, \ldots, $}
		            \State sample $(\mathbf{z}^{(1)}_t, \cdots, \mathbf{z}^{(m)}_t) \sim \mathcal{N}(0,\sigma^2 I)$
		            \State $\mathbf{g}_t(\mathbf{x}_t)  \leftarrow  \sum_{i=1}^m \mathbf{g}(\mathbf{x}_t+\mathbf{z}^{(i)}_t)$
		            \State $\mathbf{x}_{t+1} \leftarrow \mathbf{x}_t - \eta (\mathbf{g}_t(\mathbf{x}_t) + \xi_t), \qquad \xi_t \text{~uniformly~} \sim \mathbb{B}_0(r)$
		        \EndFor
		        \State \textbf{return} $\mathbf{x}_T$
	        \end{algorithmic}
\end{algorithm}
\begin{lemma}[  ]
              Let $f: \mathbb{R}^d \to \mathbb{R}$ be a bounded, $L$-continuous, $\ell$-gradient, $\rho$-Hessian Lipschitz function. Additionally, suppose that we have access to a function $g: \mathbb{R}^d \to \mathbb{R}$ such that $\norm{\nabla g-\nabla f}_{\infty}\le\nu$. 
              Then, \cite{JinL0J18}'s FPSG method needs $\tilde{\mathcal{O}}(\frac{d^3}{\epsilon^4})$ evaluations of $\nabla g$ to converge to an $\epsilon$-\SOSP.
        \end{lemma}
        \begin{proof}
        We will show the main steps that \cite{JinL0J18} followed in Section E of the Appendix. The first step of the proof is to define the Gaussian smoothing of function $g$ with parameter $\sigma$
        \begin{equation*}
            g_{\sigma}(\mathbf{x}) = \mathbb{E}_{\mathbf{z} \sim \mathcal{N}(0, \sigma^2 I)} g( \mathbf{x} + \mathbf{z})
        \end{equation*}
        One can show that
        \begin{align*}
            \nabla g_{\sigma}(\mathbf{x}) &= \mathbb{E}_{\mathbf{z} \sim \mathcal{N}(0, \sigma^2 I)} \nabla g( \mathbf{x} + \mathbf{z})\\
            \nabla^2 g_{\sigma}(\mathbf{x}) &= \mathbb{E}_{\mathbf{z} \sim \mathcal{N}(0, \sigma^2 I)} \nabla^2 g( \mathbf{x} + \mathbf{z})\\
        \end{align*}
        Additionally Lemma 48 of \cite{JinL0J18} tells us that the gradients and Hessians of $g_{\sigma}$ and $f$ are close to each other and that $g_{\sigma}$ is gradient Lipschitz and Hessian Lipschitz.
        \begin{itemize}
            \item $g_{\sigma}$ is $O(\ell + \frac{\nu}{\sigma})$ gradient Lipschitz and $O(\rho + \frac{\nu}{\sigma^2})$ Hessian Lipschitz.
            \item $\norm{\nabla g_{\sigma}(\mathbf{x}) - \nabla f(\mathbf{x})} \leq \mathcal{O}(\rho d \sigma^2 + \nu)$ and $\norm{\nabla^2 g_{\sigma}(\mathbf{x}) - \nabla^2 f(\mathbf{x})} \leq \mathcal{O}(\rho \sqrt{d} \sigma + \nu)$
        \end{itemize}
        Then Lemma 54 of \cite{JinL0J18} proves that a $\frac{\epsilon}{\sqrt{d}}$-\SOSP of $g_{\sigma}$ is also a $\mathcal{O}(\epsilon)$ stationary point of $f$ if 
        \begin{align*}
            \sigma &\leq \mathcal{O}(\sqrt{\frac{\epsilon}{\rho d}})\\
            \mathbf{\nu} &\leq \mathcal{O}(\frac{\epsilon}{\sqrt{d}})
        \end{align*}
        For the aforementioned choices of $\mathbf{\nu}$ and $\sigma$, $\nabla g$ is bounded
        \begin{equation*}
            \norm{\nabla g_{\sigma}(\mathbf{x})} \leq \norm{\nabla g_{\sigma}(\mathbf{x}) - \nabla f(\mathbf{x})} + \norm{\nabla f(\mathbf{x})}
            \leq \sqrt{d}\nu + L \leq \epsilon + L
        \end{equation*}
        So $g( \mathbf{x} + \mathbf{z})$ is $\mathcal{O}(\epsilon + L)$ sub-gaussian. Notice also that by replacing with the upper bounds on $\sigma$ and $\nu$ one can observe that the Lipschitz constant of  $\nabla^2 g_{\sigma}$ is $\mathcal{O}(\rho \sqrt{d})$. This is the main reason that a $\frac{\epsilon}{\sqrt{d}}$-\SOSP of $g_{\sigma}$ is required.
        
        According to Theorem 65 of \cite{JinL0J18} getting an $\epsilon$-\SOSP of $g_{\sigma}$ requires $\mathcal{\tilde{O}}(d/\epsilon^4)$ number of evaluations of $\nabla g$. So to get an $\frac{\epsilon}{\sqrt{d}}$-\SOSP of $g_{\sigma}$, one would require $\mathcal{\tilde{O}}(d^3/\epsilon^4)$ number of evaluations of $\nabla g$.
        \end{proof}
        
         Notice that the above theorem makes the technical assumption that the gradient approximator is a gradient of a function, that may not be true for standard finite differences approximators. The Lemma below for ZPSG does not have the same limitation. In contrast to FPSG, Algorithm \ref{algo:ZPSGD} works with function evaluations directly to come up with appropriate gradient evaluations. 
         \begin{algorithm}[h!]
        \caption{Zero order Perturbed Stochastic Gradient Descent (ZPSGD)}\label{algo:ZPSGD}
        \begin{algorithmic}
            \renewcommand{\algorithmicrequire}{\textbf{Input: }}
            \renewcommand{\algorithmicensure}{\textbf{Output: }}
            \Require $\mathbf{x}_0$, learning rate $\eta$, noise radius $r$, mini-batch size $m$.
            \For{$t = 0, 1, \ldots, $}
                \State sample $(\mathbf{z}^{(1)}_t, \cdots, \mathbf{z}^{(m)}_t) \sim \mathcal{N}(0,\sigma^2 I)$
                \State $\mathbf{g}_t(\mathbf{x}_t)  \leftarrow  \sum_{i=1}^m \mathbf{z}^{(i)}_t[f(\mathbf{x}_t+\mathbf{z}^{(i)}_t)- f(\mathbf{x}_t)]/(m\sigma^2)$
                \State $\mathbf{x}_{t+1} \leftarrow \mathbf{x}_t - \eta (\mathbf{g}_t(\mathbf{x}_t) + \xi_t), \qquad \xi_t \text{~uniformly~} \sim \mathbb{B}_0(r)$
            \EndFor
            \State \textbf{return} $\mathbf{x}_T$
        \end{algorithmic}
        \end{algorithm}
        \begin{lemma}
              Let $f: \mathbb{R}^d \to \mathbb{R}$ be a bounded, $L$-continuous, $\ell$-gradient, $\rho$-Hessian Lipschitz function.  Then, \cite{JinL0J18}'s ZPSG method needs $\tilde{\mathcal{O}}(\frac{d^2}{\epsilon^5})$ evaluations of $f$ to converge to an $\epsilon$-\SOSP.
        \end{lemma}   
         \begin{proof}
         We will show the main steps that \cite{JinL0J18} followed in Section A of the Appendix. The first step of the proof is to define the Gaussian smoothing of function $f$ with parameter $\sigma$
        \begin{equation*}
            f_{\sigma}(\mathbf{x}) = \mathbb{E}_{\mathbf{z} \sim \mathcal{N}(0, \sigma^2 I)} f( \mathbf{x} + \mathbf{z})
        \end{equation*}
        One can show that
        \begin{align*}
            \nabla f_{\sigma}(\mathbf{x}) &= \mathbb{E}_{\mathbf{z} \sim \mathcal{N}(0, \sigma^2 I)} \nabla f( \mathbf{x} + \mathbf{z})\\
            \nabla^2 f_{\sigma}(\mathbf{x}) &= \mathbb{E}_{\mathbf{z} \sim \mathcal{N}(0, \sigma^2 I)} \nabla^2 f( \mathbf{x} + \mathbf{z})\\
        \end{align*}
        Additionally Lemma 18 of \cite{JinL0J18} for $\nu =0$, tells us that the gradients and Hessians of $f_{\sigma}$ and $f$ are close to each other and that $f_{\sigma}$ is gradient Lipschitz and Hessian Lipschitz.
        \begin{itemize}
            \item $f_{\sigma}$ is $O(\ell)$ gradient Lipschitz and $O(\rho)$ Hessian Lipschitz.
            \item $\norm{\nabla f_{\sigma}(\mathbf{x}) - \nabla f(\mathbf{x})} \leq \mathcal{O}(\rho d \sigma^2)$ and $\norm{\nabla^2 f_{\sigma}(\mathbf{x}) - \nabla^2 f(\mathbf{x})} \leq \mathcal{O}(\rho \sqrt{d} \sigma)$
        \end{itemize}
        Based on this we can see that an $\epsilon$-\SOSP of $f_{\sigma}$ is also a $\mathcal{O}(\epsilon)$ stationary point of $f$ if 
        \begin{align*}
            \sigma &\leq \mathcal{O}(\sqrt{\frac{\epsilon}{\rho d}})
        \end{align*}
        We also need to develop a random gradient approximator of $\nabla f_{\sigma}$ given only evaluations $f$. Based on Lemma 19
        \begin{equation*}
            \nabla f_{\sigma}(\mathbf{x}) = \mathbb{E}_{\mathbf{z} \sim \mathcal{N}(0, \sigma^2 I)} \mathbf{z} \frac{f(\mathbf{x} + \mathbf{z}) - f(\mathbf{x}) }{\sigma^2}
        \end{equation*}
        Let us define
        \begin{equation*}
            g( \mathbf{x};\mathbf{z}) = \mathbf{z} \frac{f(\mathbf{x} + \mathbf{z}) - f(\mathbf{x}) }{\sigma^2}
        \end{equation*}
        Lemma 24 shows that $g$ is $\frac{B}{\sigma}$ subgaussian where $B$ is the upper bound on $\abs{f (\mathbf{x})}$ (it exists since $f$ is bounded). Replacing with the upper bound on $\sigma$, it turns out that $g$ is $\mathcal{O}(B\sqrt{\frac{\epsilon}{\rho d}})$ subgaussian. This dependence on $d$ and $\epsilon$ is the main reason of the slowdown in this case.
        
        According to Theorem 65 getting an $\epsilon$-\SOSP of $f_{\sigma}$ requires $\mathcal{\tilde{O}}(d^2/\epsilon^5)$ number of evaluations of $g$. Each evaluation of $g$ requires 2 evaluations of $f$. 
        \end{proof}
        
         \clearpage
\shadowbox{
\begin{minipage}[c]{5in}
In the next section, we show the complete proof of our first main result.
We will use the Stable Manifold Theorem (SMT) to prove that zero-order approximate gradient descent (AGD) avoids strict saddle points. 
\end{minipage}
}
\section{Approximate Gradient Descent Detailed proofs}

\shadowbox{
\begin{minipage}[c]{5in}
Our first two lemmas prove the equivalence between the first order stationary points of $f$ and the fixed points of the AGD. Additionally we show that saddle points of the objective function  correspond exactly to the unstable fixed of the proposed zero order method. Finally we show that for sufficiently small size-step the dynamical system is diffeomorphism. This critical property will allow us to generalize the consequences of SMT from a local region around a saddle point to the global domain.
\end{minipage}
}

    \subsection{Avoiding strict saddle points}
        \begin{lemma}
	   Assume that $g_0$ is an $(L,B,c)$ well behaved function. If $\beta < \frac{1}{B}$ and $\eta < \frac{1}{L}$ for every strict saddle point $\mathbf{x}^*$ of $f$ and we have that $\binom{\mathbf{x}^*}{0}$ is not a stable fixed point of $g_0$. Additionally, these are the only unstable fixed points of $g_0$.
    \end{lemma}
    \begin{proof}
	    For $h=0$ and at a strict saddle $\mathbf{x}^*$, we will calculate the general differential of $g_0$. 
        \begin{align*}
            \mathrm{D}g_0\binom{\mathbf{x}^*}{0}
            =&\begin{pmatrix}I-\eta  D_x q_x(\mathbf{x}^*,0)& -\eta  D_h q_x(\mathbf{x}^*,0)\\ 0&\beta \frac{\partial q_h(0)}{\partial h}  \end{pmatrix}\\
            =&\begin{pmatrix}I-\eta \nabla^2 f(\mathbf{x}^*)& -\eta  D_h q_x(\mathbf{x}^*,0)\\ 0&\beta \frac{\partial q_h(0)}{\partial h} \end{pmatrix}
        \end{align*}
        with eigenvalues $\beta \frac{\partial q_h(0)}{\partial h},( 1- \eta \lambda_i)$ , where $\lambda_i $ are eigenvalues of $\nabla^2 f(\mathbf{x}^*)$. Since $\mathbf{x}^*$ is a strict saddle, then there is at least one eigenvalue $\lambda_i<0$, and $1-\eta \lambda_i >1$.  Thus $\binom{\mathbf{x}^*}{0}$ is an unstable fixed point of $g_0$.
        To prove that these are the only unstable fixed points, observe that $\beta \frac{\partial q_h(0)}{\partial h} \in (0,1)$ so the only way $\mathrm{D}g_0\binom{\mathbf{x}^*}{0}$ has an eigenvalue greater than 1 is for some $\lambda_i$ to be negative and therefore $\mathbf{x}^*$ should be a strict saddle. 
    \end{proof}
    
    For the sake of completeness here we provide an extra lemma that proves the equivalence between the first order stationary points of $f$ and the fixed points of $g_0$.
    
    \begin{lemma}
    \label{lemma:fixed-points}
	   Assume that $g_0$ is an $(L,B,c)$-well-behaved function for a function $f$ with $\beta < \frac{1}{B}$. Then for each first order stationary point of $f$ $\mathbf{x}^*$, $\binom{\mathbf{x}^*}{0}$ is a fixed point of $g_0$. Additionally $g_0$ has no other fixed points.
    \end{lemma}
    \begin{proof}
        For $\beta < \frac{1}{B}$ we have that $g_h = \beta q_h(h)$ is a contraction since its Lipschitz constant is less than one. So the only fixed point of $g_h$ is 0. Therefore for $h \neq 0$ no point $\binom{\mathbf{x}}{h}$ is a stable point.
	    Now for $h=0$ we get that $q_x(\mathbf{x},h)= \nabla f(\mathbf{x})$ so we have
	    \begin{equation}
	        \mathbf{x}_{k+1} = \mathbf{x}_{k} - \eta \nabla f (\mathbf{x}_{k}) 
	    \end{equation}
	    So $\mathbf{x}$ is a fixed point if and only if $\nabla f  (\mathbf{x}) = 0$. Combining this with the requirement that all fixed points of $g_0$ have $h=0$ proves the lemma.
    \end{proof}

    In order to prove Theorem \ref{thm:no-strict-exp} we also have  to prove the diffeomorphism property of $g_0$.
    \begin{lemma}
	    If $g_0$ is an $(L,B,c)$ well behaved function and $\eta < \frac{1}{L}$, then $\det(\mathrm{D}g_0 (\cdot)) \neq 0$.
	    \label{lemma:invertible}
	\end{lemma}
    
    \begin{proof}
        Let
        \begin{equation}
            \mathcal{K} =\mathrm{D}_x q_x(\mathbf{x},h)
        \end{equation}
        By straightforward calculation
        \begin{align*}
            \mathrm{D}g_0\binom{x}{h} = 
            \begin{pmatrix}I - \eta\mathcal{K} &  - \eta\mathrm{D}_h q_x(\mathbf{x},h) \\ 0 &\beta \frac{\partial q_h(h)}{\partial h} \end{pmatrix}
        \end{align*}
        Given that $g(\cdot, h)$ is $L$-Lipschitz for all $h \in \mathbb{R}$, we have that $\norm{\mathcal{K}}_2\le  L$.  Clearly we have that $\det(I-\eta \mathcal{K}) \neq 0$ since $\norm{I-\eta \mathcal{K}}_2 \geq 1-\eta L > 0$. Finally we have that
        \begin{align*}
            \det(\mathrm{D}g_0\binom{x}{h} ) =\beta \frac{\partial q_h(h)}{\partial h} \det(I-\eta \mathcal{K}) \neq 0.
        \end{align*}
    \end{proof}
    
\shadowbox{
\begin{minipage}[c]{5in}
A straightforward application of result of \cite{lee2017first} and SMT will yields a saddle-avoidance lemma following kind :
\begin{center}\textit{
	       Let $X_f^*$ be the set of the strict saddle points of $f$, $\eta<\frac{1}{L}$ and $\beta<\frac{1}{B}$. Then it holds:
	       $\Pr\Big(\{\binom{\mathbf{x}_0}{h_0}:\displaystyle\lim_{k\to\infty} \mathbf{x}_k\in X_f^* \}\Big)=0$
	       }.
\end{center}
Notice that the random choice would be both on $\mathbf{x}_0,h_0$. \textbf{In the following subsection we will prove that a stronger result where the random initialization refers only to the $\mathbf{x}_0$'s domain is surprisingly possible via a new refinement of SMT}:
\begin{center}
$\forall h_0 \in \mathbb{R} : \quad \displaystyle\Pr(\displaystyle\lim_{k \to \infty} \mathbf{x}_k = \mathbf{x}^*) = 1$ 
\end{center}
\end{minipage}
}
\shadowbox{
\begin{minipage}[c]{5in}
    Let us first describe our general strategy for proving this refinement:
       \begin{enumerate}
           \item We will restate the Stable Manifold Theorem and understand its implications.\\(Section \ref{section:smt-part1})
           \item We will study the structure of the eigenvalues of $\mathrm{D}g_0$ at fixed points of $g_0$.\\(Section \ref{section:smt-part2})
           \item We will show how this affects the projections to the stable and unstable eignespaces of $\mathrm{D}g_0$.\\(Section \ref{section:smt-part3})
           \item Finally we will see how this enables us to study the dimension of the stable manifold when $h_0$ is fixed.\\(Section \ref{section:smt-part4})
       \end{enumerate}
\end{minipage}
}   
    
    \subsection{A Refinement of the Stable Manifold Theorem}

    \subsubsection{Understanding the Stable Manifold Theorem}\label{section:smt-part1}
        \begin{theorem}[Theorem III.2 \& III.7 of \cite{shub1987global}]
            Let $\mathbf{p}$ be a fixed point for the $\mathcal{C}^r$ local diffeomorphism $h : U \to R^n$ where $U \subset R^n$ is an open neighborhood of $p$ in $R^n$ and $r \ge 1$. Let $E_s\oplus E_c\oplus E_u$ be the invariant splitting of $R^n$ into generalized eigenspaces of $Dh(p)$ corresponding to eigenvalues of absolute value less than one, equal to one, and greater than one. To the $Dh(p)$ invariant subspace $E_s \oplus  E_c$ there is an associated local $h$ invariant embedded disc $W_{sc}^{loc}$ which is the graph of a $\mathcal{C}^r$ function $r:E_s \oplus  E_c\to E_u$, and ball $B$ around $p$ such that:
            \[
                h(W_{sc}^{loc}) \cap B \subset W_{sc}^{loc}.\text{ If $h^n(\mathbf{x}) \in B$ for all $n \ge 0$, then $\mathbf{x}\in W_{sc}^{loc}$}\]
            \label{thm:stable-manifold-revisited}
       \end{theorem}
    We will give some intuition on how the Stable Manifold Theorem restricts the dimensionality of the stable manifold. It essentially boils down to restricting the dimensionality of the manifold $W_{sc}^{loc}$. Let us have a $\mathbf{x} \in U$, then this can be decomposed in two vectors $\mathbf{x}_{sc}$ and $\mathbf{x}_u$, the projection of $x$ to  $E_s \oplus E_c$ and $E_u$ respectively. Thus by the construction of  $W_{sc}^{loc}$ in the proof of the Stable Manifold theorem, we know that there is a function $r: E_s \oplus E_c \to E_u$  such that if $\mathbf{x} \in W_{sc}^{loc}$ then $(\mathbf{x}_{sc},\mathbf{x}_u)\in graph(r)$, or equivalently  it holds that $\mathbf{x}_u = r (\mathbf{x}_{sc})$. By the construction of $r$, $r$ is smooth so now $dim(W_{sc}^{loc}) = dim(\mathrm{graph}(r)) = dim(E_s \oplus E_c)$. To understand why the last statement is true, the interested reader can look at example 5.14 of \cite{loring2008introduction}.
    \clearpage
    \subsubsection{Eigenvalues of the Jacobian at fixed points}\label{section:smt-part2}
    Our main tool for understanding the structure of the eigenvalues of $\mathrm{D}g_0$ at fixed points of $g_0$ is comparing it and contrasting it with its first order counterpart, gradient descent. Here is the dynamical system of gradient descent:
    \begin{equation*}
                \mathbf{x}_{k+1} = g_1(\mathbf{x}_{k})=\mathbf{x}_{k}-\eta \nabla f(\mathbf{x}_{k})
                \label{first order}
    \end{equation*}
    Now let us pick a fixed point of $f$, $\mathbf{x}^*$. Then
    \begin{align*}
            \mathrm{D}g_1(\mathbf{x}^*)=I-\eta\nabla^2 f(\mathbf{x}^*)
    \end{align*}
    is a symmetrical matrix for the $C^2$ function $f$. Then we can write down its real orthonormal eigenvectors $\{\mathbf{v}_i \}_{i=1}^d$. Without loss of generality we can reorder them so that the $k$ first eigenvectors correspond to eigenvalues less than one, the next $s$ correspond to eigenvalues that are equal to one and and the last ones correspond to eigenvalues that are larger than one in absolute value. Based on this separation between the eigenvectors, we can now define the following three vector spaces
    \begin{align*}
        E_s^{g_1} & = [\{\mathbf{v}_1,\cdots,\mathbf{v}_k\}] \\
        E_c^{g_1} & =[\{\mathbf{v}_{k+1},\cdots,\mathbf{v}_{k+s}\}]\\
        E_u^{g_1} & =[\{\mathbf{v}_{k+s+1},\cdots,\mathbf{v}_d\}]
    \end{align*}
    Then we can prove the following interesting lemma
    \begin{lemma}            \label{lemma:eigenvalues-part1}
    If $\mathbf{v}$ is eigenvector of $Dg_1(x^*)$ then $\binom{\mathbf{v}}{0}$ is eigenvector of $Dg_0\binom{\mathbf{x}^*}{0}$ with the same eigenvalue.
    \end{lemma}
    \begin{proof}
    By straightforward calculation
    \begin{align*}
            \mathrm{D}g_0\binom{\mathbf{x}^*}{0}
            =&\begin{pmatrix}I-\eta  D_x q_x(\mathbf{x}^*,0)& -\eta  D_h q_x(\mathbf{x}^*,0)\\ 0&\beta \frac{\partial q_h(0)}{\partial h}  \end{pmatrix}\\
            =&\begin{pmatrix}I-\eta \nabla^2 f(\mathbf{x}^*)& -\eta  D_h q_x(\mathbf{x}^*,0)\\ 0&\beta \frac{\partial q_h(0)}{\partial h} \end{pmatrix}\\
            =&\begin{pmatrix}\mathrm{D}g_1(\mathbf{x}^*) & -\eta  D_h q_x(\mathbf{x}^*,0)\\ 0&\beta \frac{\partial q_h(0)}{\partial h} \end{pmatrix}
            \label{differential-zero-order}
        \end{align*}
    Indeed if $\mathbf{v}$ is eigenvector of $Dg_{1}(\mathbf{x^*})$ with eigenvalue $\lambda$ then 
    \[
    \mathrm{D}g_0\binom{\mathbf{x}^*}{0}\binom{\mathbf{v}}{0}=
    \begin{pmatrix}\mathrm{D}g_1(\mathbf{x}^*)& -\eta  \mathrm{D}_h q_x(\mathbf{x}^*,0)\\ 0&\beta \frac{\partial q_h(0)}{\partial h} \end{pmatrix}\binom{\mathbf{v}}{0}=
    \binom{\lambda \mathbf{v}}{0}=\lambda \binom{\mathbf{v}}{0}
    \] 
    \end{proof}
    Now we now the form of the $d$ out of the $d+1$ generalized eigenvalues of $\mathrm{D}g \binom{\mathbf{x}^*}{0}$. There must be at least one more generalized eigenvector along with its corresponding eigenvalue. It is known that generalized eigenvectors span the whole space. But so far all the eigenvectors have a zero in the last coordinate. So the last generalized eigenvector must have a non-zero value in the last coordinate. Without loss of generality we can assume that the last coordinate is 1. So the vector will be of the form $\binom{\tilde{\mathbf{v}}}{1}$. We would like to determine its corresponding eigenvalue.
    \begin{lemma} \label{lemma:eigenvalues-part2}
    The eigenvalue of $\mathrm{D}g_0 \binom{\mathbf{x}^*}{0}$ that corresponds to $\binom{\tilde{\mathbf{v}}}{1}$ is $\beta\frac{\partial q_h(0)}{\partial h}$
    \end{lemma}
    \begin{proof}
    Since the last row of $\mathrm{D}g_0 \binom{\mathbf{x}^*}{0}$ contains only one non-zero element,  we know that the characteristic polynomial $p_0$ of $\mathrm{D}g$
    can be written as
    \begin{equation*}
        \det(\mathrm{D}g_0 \binom{\mathbf{x}^*}{0}-\lambda  I_{d+1\times d+1})= \det(\mathrm{D}g_1 (\mathbf{x}^*)-\lambda  I_{d\times d})\det(\beta \frac{\partial q_h(0)}{\partial h}-\lambda)
    \end{equation*}
    Given that all the other eigenvalues cover the roots of the first term, we know that the last eigenvalue is $\beta\frac{\partial q_h(0)}{\partial h}$.
    \end{proof}
	By assumption we know that $0 < \beta\frac{\partial q_h(0)}{\partial h} < 1$. Thus the last generalized eigenvector corresponds to a stable eigenvalue. Now we can write down the following
	\begin{align} \label{eq:spaces}
	    E_s^{g} & = \left[\left\{ \binom{\mathbf{v}_1}{0},\cdots,\binom{\mathbf{v}_k}{0}, \binom{\tilde{\mathbf{v}}}{1}\right\} \right] \nonumber\\
        E_c^{g} & =\left[\left\{\binom{\mathbf{v}_{k+1}}{0},\cdots,\binom{\mathbf{v}_{k+s}}{0}\right\} \right]\\
        E_u^{g} & =\left[\left\{\binom{\mathbf{v}_{k+s+1}}{0},\cdots,\binom{\mathbf{v}_d}{0}\right\} \right] \nonumber
	\end{align}
	\subsubsection{Projections to stable and unstable eigenspaces of the Jacobian}\label{section:smt-part3}
	In this paragraph we want to learn more about the projection to the stable and unstable eigenspaces of  $\mathrm{D}g$. Specifically for any vector $\binom{\mathbf{x}}{h}$, there are unique $\mathbf{x}_{sc}^{g_0}$, $\mathbf{x}_u^{g_0}$, $h_s$, $h_u$ such that
	\begin{align*}
	    \binom{\mathbf{x}}{h} = \binom{\mathbf{x}_{sc}^{g_0}}{h_s} &+ \binom{\mathbf{x}_u^{g_0}}{h_u}\\
        \binom{\mathbf{x}_{sc}^{g_0}}{h_s} \in E_s^{g_0} \oplus E_c^{g_0}  &\text{ and } \binom{\mathbf{x}_u^{g_0}}{h_u} \in E_u^{g_0}
	\end{align*}
	Let us compute these projections. Given that the generalized eigenvectors span the whole space, we have that there are unique $\lambda_i \in \mathbb{R}$ such that
    \begin{align*}
        \binom{\mathbf{x}}{h} = \sum_{i=1}^n \lambda_i \binom{\mathbf{v}_i}{0} + \lambda_{n+1} \binom{\tilde{\mathbf{v}}}{1} \Leftrightarrow\\
        \lambda_{n+1} = h \text{ and } \mathbf{x} = \sum_{i=1}^n \lambda_i \mathbf{v}_i + h \tilde{\mathbf{v}} \Leftrightarrow\\
        \lambda_{n+1} = h \text{ and } \mathbf{x} - h \tilde{\mathbf{v}} = \sum_{i=1}^n \lambda_i \mathbf{v}_i  \Leftrightarrow\\
        \lambda_{n+1} = h \text{ and } \lambda_i = \langle \mathbf{x} - h \tilde{\mathbf{v}}, \mathbf{v}_i \rangle
    \end{align*}
    Since $\mathbf{v}_i$ are orthogonal as eigenvectors of a symmetrical matrix. We can now find the vectors and values $\mathbf{x}_{sc}^{g_0}$, $\mathbf{x}_u^{g_0}$, $h_s$, $h_u$
    \begin{align*}
        \mathbf{x}_{sc}^{g_0} &= \sum_{i=1}^{k+\ell} \lambda_i \mathbf{v}_i + h \tilde{\mathbf{v}} \\
        &= \sum_{i=1}^{k+\ell} \langle \mathbf{x} - h \tilde{\mathbf{v}}, \mathbf{v}_i \rangle \mathbf{v}_i + h \tilde{\mathbf{v}} \\
        &= \sum_{i=1}^{k+\ell} \langle \mathbf{x} , \mathbf{v}_i \rangle \mathbf{v}_i + h \left(\tilde{\mathbf{v}} - \sum_{i=1}^{k+\ell} \langle \tilde{v}, \mathbf{v}_i \rangle \mathbf{v}_i  \right)\\
        \mathbf{x}_u^{g_0} &= \sum_{i=k+\ell+1}^{n} \lambda_i \mathbf{v}_i  \\
        &= \sum_{i=k+\ell+1}^{n}  \langle \mathbf{x} - h \tilde{\mathbf{v}}, \mathbf{v}_i \rangle \mathbf{v}_i \\
        &= \sum_{i=k+\ell+1}^{n}  \langle \mathbf{x} , \mathbf{v}_i \rangle \mathbf{v}_i - h \sum_{i=k+\ell+1}^{n}  \langle \tilde{\mathbf{v}}, \mathbf{v}_i \rangle \mathbf{v}_i \\
        h_s &= h \text { and } h_u = 0
    \end{align*}
	
	Once again we will compare and contrast with the first order case. Equivalently for every vector $\mathbf{x}$ there are unique $\mathbf{x}_{sc}^{g_1}$, $\mathbf{x}_u^{g_1}$ such that
	\begin{align*}
	    \mathbf{x} = \mathbf{x}_{sc}^{g_1} &+ \mathbf{x}_u^{g_1} \\
        \mathbf{x}_{sc}^{g_1} \in E_s^{g_1} \oplus E_c^{g_1}  &\text{ and } \mathbf{x}_u^{g_1} \in E_u^{g_1}
	\end{align*}
	Let us define
    \begin{align} \label{eq:q_vec}
        \mathbf{q} &= \tilde{\mathbf{v}} - \sum_{i=1}^{k+\ell} \langle \tilde{\mathbf{v}}, \mathbf{v}_i \rangle \mathbf{v}_i \nonumber \\
                      &= \sum_{i=k+\ell+1}^{n}  \langle \tilde{\mathbf{v}}, \mathbf{v}_i \rangle \mathbf{v}_i
    \end{align}
    Then clearly
    \begin{align} \label{eq:transform}
        \mathbf{x}_{sc}^{g_0} &= \mathbf{x}_{sc}^{g_1} + h  \mathbf{q} \nonumber \\
        \mathbf{x}_{u}^{g_0}  &= \mathbf{x}_{u}^{g_1}  - h  \mathbf{q}  \\
        h_{sc} &= h \nonumber\\
        h_u &= 0 \nonumber
    \end{align}
	\subsubsection{Restricting the dimension of the stable manifold for fixed initial \texorpdfstring{$h$}{h}}\label{section:smt-part4}
	In this paragraph we are ready to finally prove Theorem $\ref{thm:no-strict-exp}$.
	\begin{theorem}[Theorem \ref{thm:no-strict-exp} restated]
            Let $g_0$ be a $(L,B,c)$-well-behaved function for function $f$. Let $X_f^*$ be the set of strict saddle points of $f$. Then if $\eta< \frac{1}{L}$ and $\beta < \frac{1}{B}$: \[\forall h_0\in \mathbb{R}: \mu(\{ \mathbf{x}_0: \lim_{k \to \infty} \mathbf{x}_k \in X_f^* \}) = 0\]
        \end{theorem}
        \begin{proof}
         Without loss of generality let us have a fixed $h = h_0$. Let us define $M_{h_0}$ as 
        \begin{align*}
            M_{h_0} = \{\mathbf{x}_0 \in \mathbb{R}^n :  \lim_{k \to \infty} g_0^k(\mathbf{x}_0, h_0) = (\mathbf{x}^*,0)  \text{ and } \mathbf{x}^* \in X_f^*\}
        \end{align*}
        We want to prove that the set $M$ has measure 0. Let us apply the Stable Manifold Theorem on $g_0$ for all fixed points $\mathbf{p}=(\mathbf{x}^*,0)\in X_f^*\times \{0\}$. Let $B_{\mathbf{p}}$, $W_{sc,\mathbf{p}}^{loc}$ be the ball and the corresponding manifold derived by Theorem \ref{thm:stable-manifold-revisited}. We consider the union of those balls $\mathcal{B}=\bigcup B_{p}$. The following property for $\mathbb{R}^N$ holds:
        \begin{theorem*}[Lindel\"{o}f’s lemma]
            For every open cover there is a countable subcover.
        \end{theorem*}
        Therefore due to Lindel\"{o}f’s lemma, we can find a countable subcover for $\mathcal{B}$, i.e., there exists a countable family of fixed-points $\mathbf{p}_0, \mathbf{p}_1, \cdots $ such that  $\mathcal{B}=\bigcup_{m=0}^{+\infty} B_{p_{m}}$.
        Once again, based on Theorem \ref{thm:stable-manifold-revisited}, if starting from $\mathbf{x}_0$ one converges to an unstable fixed point then it holds that
        \begin{align*}
            \mathbf{x}_0 \in M_{h_0} &\Rightarrow \exists m,t_0 : \forall t  \geq t_0 \ (\mathbf{x}_t, h_t) = g_0^t(\mathbf{x}_0, h_0) \text{ and } (\mathbf{x}_t, h_t) \in B_{\mathbf{p}_m} \\
            &\Rightarrow \exists m,t_0 : (\mathbf{x}_{t_0}, h_{t_0}) = g_0^{t_0}(\mathbf{x}_0, h_0) \text{ and } (\mathbf{x}_{t_0}, h_{t_0}) \in W_{sc,\mathbf{p}_m}^{loc}\\
        \end{align*}
        Let us define
        \begin{align*}
            U_t^m = \{\mathbf{x}_0 \in \mathbb{R}^d : (\mathbf{x}_t, h_t) = g_0^{t}(\mathbf{x}_0, h_0) \text{ and } (\mathbf{x}_t, h_t) \in W_{sc,\mathbf{p}_m}^{loc} \}
        \end{align*}
        Therefore we have
        \begin{align*}
            M_{h_0} \subseteq \bigcup_{m = 0}^\infty\bigcup_{t = 0}^\infty U_t^m \\
        \end{align*}
        Now it suffices to prove that all $U_t^m$ sets have zero measure. Let us first prove the following lemma as a stepping stone.
        \begin{lemma*}
        Let us define the following set of points
        \begin{align*}
            R_h^m = \{\mathbf{x} \in \mathbb{R}^d : (\mathbf{x}, h) \in W_{sc,\mathbf{p}_m}^{loc} \}
        \end{align*}
        Then $dim(R_h^m) < d$.
        \end{lemma*}
        \begin{proof}
        Based on our discussion on the Stable Manifold Theorem, we  know that there is a smooth function $r: E_s^{g} \oplus E_c^{g} \to E_u^{g}$ such that
        \begin{align*}
            \binom{\mathbf{x}}{h} \in W_{sc,\mathbf{p}_m}^{loc} &\Rightarrow \binom{\mathbf{x}_{u}^{g_0}}{h_s} = r (\mathbf{x}_{sc}^{g_0}, h_u)
        \end{align*}
        where $\mathbf{x}_{u}^{g}$, $\mathbf{x}_{sc}^{g}$ , $h_s$ and $h_u$ the components of the projections to $E_s^{g_0} \oplus E_c^{g_0}$ and $ E_u^{g_0}$ as defined in the Equations of \ref{eq:spaces}. Now using our analysis in the Equations of \ref{eq:transform}
        \begin{align*}
            \binom{\mathbf{x}}{h} \in W_{sc,\mathbf{p}_m}^{loc} &\Rightarrow \binom{\mathbf{x}_{u}^{g_1} - h \mathbf{q}}{0} = r (\mathbf{x}_{sc}^{g_1}+ h \mathbf{q}, h)
        \end{align*}
        where $\mathbf{q}$ is the vector we defined in Equation \ref{eq:q_vec}.
        Let $\prod$ be the projection that for each $\binom{\mathbf{x}}{h} \in \mathbb{R}^{d+1}$ returns $\mathbf{x}$. Then we can define the following smooth function
        \begin{align*}
            r_{h}' : E_s^{g_1} \oplus E_c^{g_1} \to E_u^{g_1} \text{ } r_{h}'(\mathbf{x}) =  h  \mathbf{q} + \prod r(\mathbf{x} + h \mathbf{q}, h).
        \end{align*}
        Using the $\{ \mathbf{v}_i\}_{i=1}^n$ as a basis we can write
        \begin{align*}
            \binom{\mathbf{x}}{h} \in W_{sc,\mathbf{p}_m}^{loc} \Rightarrow  \mathbf{x}_{u}^{g_1} = r_{h}' (\mathbf{x}_{sc}^{g_1}) \Rightarrow \mathbf{x} \in \mathrm{graph}(r_{h}')
        \end{align*}
        Therefore $dim(R_h^m) \leq dim(E_s^{g_1} \oplus E_c^{g_1}) < d$ since $\mathbf{p}_m$ corresponds to an unstable fixed point of $g_1$.
        \end{proof}
        Then we can prove the following lemma
        \begin{lemma}
            The measure of $U_t^m$ is zero.
        \end{lemma}
        \begin{proof}
        We will do this by contradiction. Let us assume that $U_t$ has non-zero measure. Let us define 
        \begin{align*}
            W_0^m &= \{ \mathbf{x} \in \mathbb{R}^n : x \in U_t^m \}\\
            W_1^m &= \{ \mathbf{x} \in \mathbb{R}^n : x \in g_0(W_1^m, h_1) \}\\
            \vdots\\
            W_t^m &= \{ \mathbf{x} \in \mathbb{R}^n : x \in g_0(W_{t-1}^m, h_{t-1}) \}\\
        \end{align*}
        Given that $g( \cdot, h_i)$ is a diffeomorphism for all $i$, we have that $W_i$ has non zero measure. 
        Observe that
        \begin{align*}
            W_t^m \subseteq R_{h_t}^m
        \end{align*}
        and so $dim(W_t^m) < d$ and $W_t^m$ has measure zero leading to a contradiction.
    \end{proof}
        Since the countable union of zero measure sets is zero measure we clearly have that $M_{h_0}$ has measure zero as requested.
    \end{proof}

\shadowbox{
\begin{minipage}[c]{5in}
In the previous section, we provided sufficient conditions to avoid convergence to strict saddle points. These results are meaningful however only if $\displaystyle\lim_{k \to \infty} \mathbf{x}_k = \mathbf{x}^*$. Thus in order to complete the proof of \ref{thm:second-order}, in the following section we will 
provide sufficient conditions such that the dynamic system of AGD converges.
\end{minipage}
}   
\clearpage
\subsection{Convergence} 

    We will refer to the error of the gradient approximation as
    \begin{equation*}
            \boldsymbol{\varepsilon}_k = q_x(\mathbf{x}_k,h_k) - \nabla f(\mathbf{x}_k).
    \end{equation*}

\shadowbox{
\begin{minipage}[c]{5in}
    In order to prove the convergence firstly we establish a lower bound for the decrease of the function that is connected with the norm of the gradient and its approximation error (Lemma \ref{lemma:step-convergence}). We also prove that our scheme yields to an exponential decrease of that error (Lemma \ref{lemma:small-error}). Given those lemmas we can prove an exact and an $\epsilon-$first order stationary convergence theorem.
\end{minipage}
}   

    \begin{lemma}[Lemma \ref{lemma:step-convergence} restated]
        Suppose that $g_0$ is a $(L,B,c)$-well-behaved function for a $\ell$-gradient Lipschitz function $f$.
        If $\eta \leq \frac{1}{\ell}$ then we have that
        \begin{equation}
            f(\mathbf{x}_{k+1}) \le f(\mathbf{x}_{k}) -\frac{\eta}{2}\left(\norm{\nabla f(\mathbf{x}_k)}^2-\norm{\boldsymbol{\varepsilon}_k}^2 \right)
        \end{equation}
    \end{lemma}
    \begin{proof}
        \begin{align*}
    f(\mathbf{x}_{k+1})&\le f(\mathbf{x}_{k}) +\nabla f(\mathbf{x}_k)^\top(\mathbf{x}_{k+1}-\mathbf{x}_{k})+\frac{\ell}{2}\norm{\mathbf{x}_{k+1}-\mathbf{x}_{k}}^2\\
    \\&\le f(\mathbf{x}_{k}) -\eta\nabla f(\mathbf{x}_k)^\top q_x(\mathbf{x}_k,h_k)+\frac{\eta^2\ell}{2}\norm{q_x(\mathbf{x}_k,h_k)}^2\\
    \\&\le f(\mathbf{x}_{k}) -\eta\nabla f(\mathbf{x}_k)^\top (\nabla f(\mathbf{x}_k)+\boldsymbol{\varepsilon}_k)+\frac{\eta^2\ell}{2}\norm{\nabla f(\mathbf{x}_k)+\boldsymbol{\varepsilon}_k}^2\\
    \\&\le f(\mathbf{x}_{k}) -\eta\nabla f(\mathbf{x}_k)^\top (\nabla f(\mathbf{x}_k)+\boldsymbol{\varepsilon}_k)+\frac{\eta}{2}\norm{\nabla f(\mathbf{x}_k)+\boldsymbol{\varepsilon}_k}^2\\
    \\&\le f(\mathbf{x}_{k}) -\frac{\eta}{2}\left(\norm{\nabla f(\mathbf{x}_k)}^2-\norm{\boldsymbol{\varepsilon}_k}^2 \right)
        \end{align*}
    \end{proof}

    \begin{lemma}[Exponentially Decreasing $\boldsymbol{\varepsilon}_k$]
        Suppose that $g_0$ is a $(L,B,c)$-well-behaved function for a function $f$. Then we have that
        \begin{equation*}
            \norm{\boldsymbol{\varepsilon}_k} \leq c \abs{h_0} (\beta B)^k
        \end{equation*}
        \label{lemma:small-error}
    \end{lemma}
    \begin{proof}
        Since $q_h$ is $B$-Lipschitz
        \begin{equation*}
            \abs{h_{k+1}} = \abs{\beta q_h(h_k) - \beta q_h(0)} \leq \beta B \abs{h_k}
        \end{equation*}
        Therefore we have that
        \begin{equation*}
            \abs{h_k} \leq (\beta B)^k \abs{h_0} 
        \end{equation*}
        Based on property 3 of the $(L,B,c)$-well-behaved function we have that
        \begin{equation*}
           \norm{\boldsymbol{\varepsilon}_k} = \norm{q_x(\mathbf{x}_k, h_k) - \nabla f(\mathbf{x}_k) } \leq c \abs{h_k} = (\beta B)^k \abs{h_0}  
        \end{equation*}
    \end{proof}
    
    Now we are ready to start our proof for the convergence to the first order stationary points.
    \begin{theorem}[ Theorem \ref{theorem:exact-stationary-points}  Restated]       
        Suppose that $g_0$ is a $(L,B,c)$-well-behaved gradient function for a $\ell$-gradient Lipschitz function $f$. Let $\eta \leq \frac{1}{\ell}$, $\beta < \frac{1}{B}$. Then if $f$ is lower bounded
        \begin{equation*}
            \lim_{k \to \infty} \norm{\nabla f(\mathbf{x}_k)} = 0 
        \end{equation*}
    \end{theorem}

    \begin{proof}
        Applying Lemma \ref{lemma:step-convergence} repeatedly we get
        \begin{align*}
            f(\mathbf{x}_0)-f(\mathbf{x}_{k}) &\geq \frac{\eta}{2}\sum_{i=0}^{k}
            \left(\norm{\nabla f(\mathbf{x}_{i})}^2-\norm{\boldsymbol{\varepsilon}_{i}}^2 \right)\\
        \end{align*}
        We now have that
        \begin{align*}
            f(\mathbf{x}_0)-f(\mathbf{x}_{k}) + \frac{\eta}{2}\sum_{i=0}^{k}\norm{\boldsymbol{\varepsilon}_{i}}^2 &\geq \frac{\eta}{2}\sum_{i=0}^{k}\norm{\nabla f(\mathbf{x}_{i})}^2\\
            f(\mathbf{x}_0)-f(\mathbf{x}_{k}) + \frac{\eta}{2}\sum_{i=0}^{\infty}\norm{\boldsymbol{\varepsilon}_{i}}^2 &\geq \frac{\eta}{2}\sum_{i=0}^{k}\norm{\nabla f(\mathbf{x}_{i})}^2\\
            f(\mathbf{x}_0)-f(\mathbf{x}_{k}) + \frac{\eta}{2}\sum_{i=0}^{\infty}\left( c \abs{h_0} (\beta B)^i \right)^2 &\geq \frac{\eta}{2}\sum_{i=0}^{k}\norm{\nabla f(\mathbf{x}_{i})}^2\\
            f(\mathbf{x}_0)-f(\mathbf{x}_{k}) + \frac{\eta}{2} \frac{c^2 h_0^2}{1-(\beta B)^2} &\geq \frac{\eta}{2}\sum_{i=0}^{k}\norm{\nabla f(\mathbf{x}_{i})}^2\\
        \end{align*}
        Given that $f$ is lower bounded, $f(\mathbf{x}_0)-f(\mathbf{x}_{k})$ and therefore the whole left hand side is upper bounded which means the series sum in the right hand side is upper bounded. Since this is a series of non negative terms this means that the series converges and therefore
        \begin{equation*}
            \lim_{k \to \infty} \norm{\nabla f(\mathbf{x}_k)} = 0 
        \end{equation*}
    \end{proof}

\shadowbox{
\begin{minipage}[c]{5in}
For the sake of completeness, we will analyze the convergence rate to $\epsilon$-first order stationary points in this setting. This would enable us to to make a fair comparison with previous results that assume a fixed $h_k = h_0$.
Notice that the following result improves over previous work in randomized zero order gradient approximations. In \cite{nesterov2017random}, it was proved that using a randomized oracle that requires 2 function evaluations per iteration, one could get an in expectation $\epsilon$-first order stationary point after $\mathcal{O}\left(d\ell \left(f(\mathbf{x}_0) -f^*\right)/\epsilon^2\right)$ iterations. For the case of $q_x$ using $r_f$  as defined in Equation \ref{algo:forward} of the Section 3, we have just proved that with $d+1$ function evaluations per iteration we can get a $\epsilon$-first order stationary point after only $\mathcal{O}\left(\ell \left(f(\mathbf{x}_0) -f^*\right)/\epsilon^2\right)$ iterations. Thus for the same number of function evaluations up to constants, our work provides deterministic guarantees whereas \cite{nesterov2017random} provides guarantees only in expectation.
\end{minipage}
}   
    
    \begin{theorem}[$\epsilon$-first order stationary points] 
            \label{theorem:e-stationary-points} 
                Suppose that $g_0$ is a $(L,B,c)$-well-behaved gradient function for a $\ell$-gradient Lipschitz function $f$. Let $q_h(h) =h$ and $\beta =1$, $\eta = \frac{1}{\ell}$. Then if $f$ has minimum value $f^*$ and $h_0 = \frac{\epsilon}{\sqrt{2}c}$, the required number of iterations to reach a $\epsilon$-first order stationary point is
                \begin{equation*}
                    \mathcal{O}\left(\frac{\ell \left(f(\mathbf{x}_0) -f^*\right)}{\epsilon^2}\right) 
            \end{equation*}
    \end{theorem}
    
    \begin{proof}
        Applying Lemma \ref{lemma:step-convergence} repeatedly we get
        \begin{align*}
            f(\mathbf{x}_0)-f(\mathbf{x}_{k}) &\geq \frac{1}{2\ell}\sum_{i=0}^{k}
            \left(\norm{\nabla f(\mathbf{x}_{i})}^2-\norm{\boldsymbol{\varepsilon}_{i}}^2 \right)\\
        \end{align*}
        We now have that
        \begin{align*}
            f(\mathbf{x}_0)-f(\mathbf{x}_{k}) + \frac{1}{2\ell}\sum_{i=0}^{k}\norm{\boldsymbol{\varepsilon}_{i}}^2 &\geq \frac{1}{2\ell}\sum_{i=0}^{k}\norm{\nabla f(\mathbf{x}_{i})}^2\\
            f(\mathbf{x}_0)-f(\mathbf{x}_{k}) + \frac{k+1}{2\ell} \left(c \abs{h_0}\right)^2 &\geq \frac{1}{2\ell}\sum_{i=0}^{k}\norm{\nabla f(\mathbf{x}_{i})}^2\\
            \frac{\ell (f(\mathbf{x}_0)-f(\mathbf{x}_{k}))}{2(k+1)} + c^2 \abs{h_0}^2 &\geq \frac{1}{k+1}\sum_{i=0}^{k}\norm{\nabla f(\mathbf{x}_{i})}^2\\
            \frac{\ell (f(\mathbf{x}_0)-f^*)}{2(k+1)} + \frac{\epsilon^2}{2} &\geq \frac{1}{k+1}\sum_{i=0}^{k}\norm{\nabla f(\mathbf{x}_{i})}^2\\
        \end{align*}
        Choose the smallest $k_0$ such that $\frac{\ell (f(\mathbf{x}_0)-f^*)}{(k_0+1)}  \leq \epsilon^2$. Then we have
        \begin{equation*}
            \epsilon^2 \geq \frac{1}{k_0+1}\sum_{i=0}^{k_0}\norm{\nabla f(\mathbf{x}_{i})}^2
        \end{equation*}
        Since the average of the squared norms of the gradients is less than $\epsilon^2$, there should be at least one that is less or equal to $\epsilon^2$. That is there is a $k \leq k_0$ such that $\norm{\nabla f(\mathbf{x}_{k})} \leq \epsilon$. Given the definition of $k_0$ we get the iteration bound stated in the theorem.
    \end{proof}

\shadowbox{
\begin{minipage}[c]{5in}
    The last theorems give us a guarantee that the norm of the gradient is converging to zero but this is not enough to prove convergence to a single stationary point if $f$ has non isolated critical points. To establish a stronger result we prove that $\{\norm*{\nabla f(\mathbf{x}_k)}\}$ does not decrease arbitrarily quickly. 
\end{minipage}
}   

    \begin{lemma}[Sufficiently large gradients]
        Suppose that $g_0$ is a $(L,B,c)$-well-behaved function for a $\ell$-gradient Lipschitz function $f$. Then we have that
        \begin{equation*}
            \norm{\nabla f(\mathbf{x}_{k+1})} \geq (1-\eta \ell)\norm{\nabla f(\mathbf{x}_k)} -\eta \ell \norm{\boldsymbol{\varepsilon}_k}
        \end{equation*}
        \label{lemma:large-grads}
    \end{lemma}
    \begin{proof}
        \begin{align*}
            \norm{\nabla f(\mathbf{x}_{k+1})} &\geq   \norm{\nabla f(\mathbf{x}_k)}- \norm{\nabla f(\mathbf{x}_{k+1})-\nabla f(\mathbf{x}_k)}\\
            &\geq   \norm{\nabla f(\mathbf{x}_k)}- \ell \norm{\mathbf{x}_{k+1}-\mathbf{x}_k}\\
            &\geq \norm{\nabla f(\mathbf{x}_k)} - \eta \ell \norm{q_x(\mathbf{x}_k,h_k)}\\
            &\geq \norm{\nabla f(\mathbf{x}_k)} - \eta \ell \norm{\nabla f(\mathbf{x}_k) + \boldsymbol{\varepsilon}_k}\\
            &\geq \norm{\nabla f(\mathbf{x}_k)} - \eta \ell \norm{\nabla f(\mathbf{x}_k)} - \eta \ell\norm{\boldsymbol{\varepsilon}_k}\\
            &\geq (1-\eta \ell)\norm{\nabla f(\mathbf{x}_k)} - \eta \ell\norm{\boldsymbol{\varepsilon}_k}
        \end{align*}
    \end{proof}
    \clearpage
    \shadowbox{
\begin{minipage}[c]{5in}
   
    Having established the above lemma we can use the Theorem 3.2 in \cite{absil2005convergence} and we are able to provide sufficient conditions to get convergence to a single stationary point even for functions with non isolated critical points.
\end{minipage}
}  
    \begin{theorem} \label{thm:point-convergence}
    Assume that $f$ is $\ell$-gradient Lipschitz, is analytic and that it has compact sub-level sets and that $g_0$ is a $(L,B,c)$-well-behaved gradient oracle. Let $\eta < \frac{1}{2\ell}$, $\beta < \frac{1-2\eta\ell}{B}$. Then $\lim \mathbf{x}_k$ exists and is a stationary point of $f$.
    \end{theorem}
    \begin{proof}
        We will first prove that given the fact that $f$ has compact sub-level sets $\{\mathbf{x}_k\}$ is confined in compact set. Based on Lemma \ref{lemma:step-convergence} we have that for all $k \geq 0$
        \begin{equation*}
            f(\mathbf{x}_{k+1})-f(\mathbf{x}_k) \leq \frac{\eta}{2}\norm{\boldsymbol{\varepsilon}_k}^2
        \end{equation*}
        Applying this recursively and adding the inequalities
        \begin{align*}
            f(\mathbf{x}_{k+1}) &\leq f(\mathbf{x}_0) +\frac{\eta}{2} \sum_{i=0}^k\norm{\boldsymbol{\varepsilon}_i}^2\\
            &\leq f(\mathbf{x}_0) + \frac{\eta}{2} \sum_{i=0}^k\left( c \abs{h_0} (\beta B)^i \right)^2\\
            &\leq f(\mathbf{x}_0) + \frac{\eta}{2}c^2 h_0^2 \sum_{i=0}^k(\beta B)^{2i}\\
            &\leq f(\mathbf{x}_0) +\frac{\eta}{2}c^2 h_0^2 \frac{1}{1-(\beta B)^2}\\
        \end{align*}
        So clearly $\{f(\mathbf{x}_k)\}$ is bounded and therefore $\{\mathbf{x}_k\}$ stays in one of the compact sub-level sets of $f$ forever. 
        
        Let us define the following
        \begin{equation*}
                \phi_{k}(h_0) = c \abs{h_0} (\beta B)^k
        \end{equation*}
        
        We will split the proof of the theorem in two cases. For the first case we will assume that there is a $k_0 \in \mathbb{N}$ such that
        \begin{equation*}
            \norm{\nabla f(\mathbf{x}_{k_0})} \geq \phi_{k_0}(h_0)
        \end{equation*}
        Then by Lemma \ref{lemma:large-grads}
        \begin{align*}
            \norm{\nabla f(\mathbf{x}_{k_0+1})} &\geq (1-\eta \ell)\norm{\nabla f(\mathbf{x}_{k_0})}- \eta \ell\norm{\boldsymbol{\varepsilon}_{k_0}}_2\\
            &\geq  (1-\eta \ell)\phi_{k_0}(h_0) - \eta \ell\phi_{k_0}(h_0)\\
            &\geq  (1-2\eta \ell)\phi_{k_0}(h_0)\\
            &\geq  \frac{1-2\eta \ell}{\beta B} \beta B \phi_{k_0}(h_0)\\
            &\geq  \frac{1-2\eta \ell}{\beta B} \phi_{k_0+1}(h_0)\\
            &\geq  \phi_{k_0+1}(h_0)\\
        \end{align*}
        By induction we have that $\forall k \geq k_0 + 1$
        \begin{align*}
            \norm{\nabla f(\mathbf{x}_{k})} &\geq \frac{1-2\eta \ell}{\beta B}  \phi_{k}(h_0)\\
        \end{align*}
        By Lemma \ref{lemma:small-error}
        \begin{equation*}
            \frac{\norm{\nabla f(\mathbf{x}_{k})}}{\norm{\boldsymbol{\varepsilon}_{k}}}  \geq
            \left(\frac{1-2\eta \ell}{\beta B}\right) = q >1
        \end{equation*}
        At the same time
        \begin{align*}
            -\nabla f(\mathbf{x}_k)^\top (\mathbf{x}_{k+1}-\mathbf{x}_k) &= \eta  \nabla f(\mathbf{x}_k)^\top (\nabla f(\mathbf{x}_k) +\boldsymbol{\varepsilon}_k)\\
            &= \eta \norm{\nabla f(\mathbf{x}_k)}^2 + \eta \nabla f(\mathbf{x}_k)^\top \boldsymbol{\varepsilon}_k\\
            &\leq \eta \left(1 + \frac{1}{q} \right) \norm{\nabla f(\mathbf{x}_k)}^2
        \end{align*}
        Additionally using similar arguments as above
        \begin{equation*}
            \frac{-\nabla f(\mathbf{x}_k)^\top (\mathbf{x}_{k+1}-\mathbf{x}_k)}{\norm{\nabla f(\mathbf{x}_k)}\norm{(\mathbf{x}_{k+1}-\mathbf{x}_k)}}   \geq \frac{\eta \left(1 - \frac{1}{q} \right)\norm{\nabla f(\mathbf{x}_k)}^2}{\eta \left(1 + \frac{1}{q} \right)\norm{\nabla f(\mathbf{x}_k)}^2} = \frac{\left(1 - \frac{1}{q} \right)}{\left(1 + \frac{1}{q} \right)}
        \end{equation*}
        Let us define
        \begin{align*}
            c_1 &= \frac{1}{2}\left(1-\frac{1}{q}\right)\\
            c_2 &= \frac{\left(1 - \frac{1}{q} \right)}{\left(1 + \frac{1}{q} \right)}\\
        \end{align*}
        Clearly by Lemma \ref{lemma:step-convergence} we have that
        \begin{equation*}
            f(\mathbf{x}_k) - f(\mathbf{x}_{k+1}) \geq \frac{\eta}{2}\left(\norm{\nabla f(\mathbf{x}_k)}^2 - \norm{\boldsymbol{\varepsilon}_k}^2\right) \geq \frac{\eta}{2} \left(1 -\frac{1}{q^2} \right)\norm{\nabla f(\mathbf{x}_k)}^2
        \end{equation*}
        We can conclude that
        \begin{equation*}
            f(\mathbf{x}_k) - f(\mathbf{x}_{k+1}) \geq -c_1 \nabla f(\mathbf{x}_k)^\top (\mathbf{x}_{k+1}-\mathbf{x}_k) \geq 
            c_1c_2 \norm{\nabla f(\mathbf{x}_k)}\norm{(\mathbf{x}_{k+1}-\mathbf{x}_k)}
        \end{equation*}
        with $c_1c_2>0$. Moreover, $\norm{\nabla f(\mathbf{x}_k)} \geq \phi_{k}(h_0)>0$ so we do not have to worry about arriving on stationary points in finite time. Given that $f$ is analytic, we have all the necessary conditions of Theorem 3.2 in \cite{absil2005convergence} and we have ruled out the possibility of $\{\mathbf{x}_k\}$ escaping to infinity. Therefore, we can now claim that $\{\mathbf{x}_k\}$ converges.
        
        For the second case we have that for for all $k \in \mathbb{N}$ 
        \begin{equation*}
            \norm{\nabla f(\mathbf{x}_k)} < \phi_{k}(h_0).
        \end{equation*}
        We will now prove that $\{\mathbf{x}_k\}$ is a Cauchy sequence.
        \begin{align*}
            \norm{\mathbf{x}_k-\mathbf{x}_m} &\leq \sum_{i=m}^k \norm{\mathbf{x}_{i+1}-\mathbf{x}_i}\\
            &\leq \sum_{i=m}^k \norm{\eta q_x(\mathbf{x}_i, h_i)}\\
            &\leq \eta  \sum_{i=m}^k \norm{\nabla f(\mathbf{x}_i, h_i) + \boldsymbol{\varepsilon}_i }\\
            &\leq 2 \eta  \sum_{i=m}^k \phi_{i}(h_0)
        \end{align*}
        We know that $\sum_{i}^\infty \phi_{i}(h_0)$ converges so the partial sums must converge to 0. Then
        \begin{equation*}
            \lim_{m,k \to \infty} \norm{\mathbf{x}_k-\mathbf{x}_m} \leq 2\eta \lim_{m,k \to \infty} \sum_{i=m}^k \phi_{i}(h_0) = 0
        \end{equation*}
        So $\lim_{m,k \to \infty} \norm{\mathbf{x}_k-\mathbf{x}_m} =0$ and $\{\mathbf{x}_k\}$ is a Cauchy sequence bounded in a compact set and therefore it converges. 
        
        In either of the cases the limit of $\{\mathbf{x}_k\}$ is of course a stationary point.
    \end{proof}
    \clearpage            
    We can now conclude our analysis with this final theorem.
            
    \begin{theorem}[Theorem \ref{thm:second-order} restated]
         Let $f: \mathbb{R}^d \to \mathbb{R} \in C^2$ be a $\ell$-gradient Lipschitz function. Let us also assume that $f$ is analytic, has compact sub-level sets and all of its saddle points are strict. Let $g_0$ be a $(L,B,c)$-well-behaved function for $f$ with $\eta < \min\{\frac{1}{L}, \frac{1}{2\ell} \}$ and $\beta < \frac{1-2\eta \ell}{B}$. If we pick a random initialization point $\mathbf{x}_0$, then we have that for the $\mathbf{x}_k$ iterates of $g_0$
            \begin{equation*}
               \forall h_0 \in \mathbb{R} \quad \Pr(\lim_{k \to \infty} \mathbf{x}_k = \mathbf{x}^*) = 1 
            \end{equation*}
        where $\mathbf{x}^*$ is a local minimizer of $f$.
    \end{theorem}
    
    \begin{proof}
        Given the assumptions, we can apply Theorem \ref{thm:point-convergence} and get that $\lim_{k \to \infty} \mathbf{x}_k$ exists and is a stationary point of $f$. We can also apply Theorem \ref{thm:no-strict-exp} in order to guarantee that the limit is not a strict saddle of $f$ with probability 1. Given the assumption that $f$ has only strict saddles, then $\lim_{k \to \infty} \mathbf{x}_k$ is with probability 1 a local minimum of $f$.
    \end{proof} \clearpage
\section{Escaping Saddle Points Efficiently Detailed proofs}
    Before presenting the iteration complexity proof ( Theorem \ref{thm:noise} ) we will state our main  probabilistic lemma.
    \begin{lemma}\label{lem:main_saddle_restated}
        There exists an absolute constant $c_{\max}$, such that for any $f$ that is $\ell$-gradient Lipschitz and $\rho$-Hessian Lipschitz function
        and any $c \le c_{\max}$, and $\chi \ge 1$. Let $\eta, r, g_{\text{thres}}, f_{\text{thres}}, t_{\text{thres}}, h_{low}$ be calculated same way as in Algorithm \ref{algo:PAGD}. Then, if $\mathbf{x}_t$ satisfies:
        \begin{equation*}
            \norm{\nabla f(\mathbf{x}_t)} \le g_{\text{thres}} \quad \quad \text{~and~} \quad \quad \lambda_{\min}(\nabla^2 f(\mathbf{x}_t)) \le -\sqrt{\rho\epsilon}
        \end{equation*}
        Let $\tilde{\mathbf{x}}_0 = \mathbf{x}_t + \boldsymbol{\xi}$, where $\boldsymbol{\xi}$ comes from the uniform distribution over $\mathcal{B}_{\mathbf{0}}(r)$, and let $\{\tilde{\mathbf{x}}_i\}$ be the iterates of approximate gradient descent from $\tilde{\mathbf{x}}_0$ with stepsize $\eta$ and $h=h_{low}$, then with at least probability $1-\frac{d\ell}{\sqrt{\rho\epsilon}}e^{-\chi}$, we have:
        \begin{equation*}
         \exists\quad i \leq t_{\text{thres}}: f(\mathbf{x}_t)-f(\tilde{\mathbf{x}}_{i'})  \ge f_{\text{thres}}
        \end{equation*}
    \end{lemma}
    This lemma will be the ``workhorse'' which will offer the high probability guarantees of Algorithm \ref{algo:PAGD} given that substantial progress can be made in the low gradient phase.
    The proof of the above lemma is deferred to the end of this section.

    We are ready now to prove our main theorem:
    
    \begin{theorem}[Theorem \ref{thm:noise} restated]
        There exists absolute constant $c_{\max}$ such that: if $f$ is  $\ell$-gradient Lipschitz and $\rho$-Hessian Lipschitz,  then for any $\delta>0, \epsilon \le \frac{\ell^2}{\rho}, \Delta_f \ge f(\mathbf{x}_0) - f^\star$, and constant $c \le c_{\max}$, with probability $1-\delta$, the output of $\text{PAGD}(x_0, \ell, \rho, \epsilon, c, \delta, \Delta_f)$ will be $\epsilon$-\SOSP, and terminate in iterations:
        \begin{equation*}
        \mathcal{O}\left(\frac{\ell(f(\mathbf{x}_0) - f^\star)}{\epsilon^2}\log^{4}\left(\frac{d\ell\Delta_f}{\epsilon^2\delta}\right) \right)
        \end{equation*}
    \end{theorem}

    \begin{proof}
        Denote $\tilde{c}_{\max}$ to be the absolute constant allowed in Lemma \ref{lem:main_saddle_restated}. 
        In this theorem, we let $c_{\max} = \min\{\tilde{c}_{\max}, 3/32\}$, and choose any constant $c \le c_{\max}$.

        In this proof, that Algorithm \ref{algo:PAGD} returns a point $\mathbf{x}$ that satisfies the following condition:
        \begin{equation} \label{eq:tighter_cond}
         \norm{\nabla f(\mathbf{x})} \le g_{\text{thres}} = \frac{\sqrt{c}}{\chi^2} \cdot \epsilon, \qquad\qquad \lambda_{\min}(\nabla^2 f(\mathbf{x})) \ge - \sqrt{\rho \epsilon}
        \end{equation}
        Since $c\le1$, $\chi\ge 1$, we have $\frac{\sqrt{c}}{\chi^2} \le 1$, which implies any $\mathbf{x}$ satisfies Equation \eqref{eq:tighter_cond} is also a $\epsilon$-\SOSP.
        
        Starting from $\mathbf{x}_0$, we know if $\mathbf{x}_0$ does not satisfy Equation \ref{eq:tighter_cond}, there are only two cases:
        \begin{enumerate}
        \item $\norm*{\mathbf{z}_0}=\norm*{ q\left(\mathbf{x}_0,\frac{g_{\text{thres}}}{4c_h}\right)}> \frac{3}{4}g_{\text{thres}}$\\
        In this case, $\norm*{\nabla f(\mathbf{x}_0)}\ge \frac{g_{\text{thres}}}{2}$ and Algorithm \ref{algo:PAGD} will not add perturbation. By Lemma \ref{lemma:step-convergence}:
        \begin{equation*}
        f(\mathbf{x}_0)-f(\mathbf{x}_1) \ge  \frac{\eta}{2} \cdot (\norm{\nabla f(\mathbf{x}_0)}^2 -\norm{\boldsymbol{\varepsilon}_0}^2) 
        \end{equation*}
        where $\boldsymbol{\varepsilon}_0=q\left(\mathbf{x}_0,\frac{g_{\text{thres}}}{4c_h}\right)-\nabla f(\mathbf{x}_0)$.
        Therefore we get $\norm*{\boldsymbol{\varepsilon}_0}\le \frac{g_{\text{thres}}}{4}$
        
        \begin{equation*}
        f(\mathbf{x}_0)-f(\mathbf{x}_1) \ge  \frac{\eta}{2} \cdot (\norm{\nabla f(\mathbf{x}_0)}^2  -\norm{\boldsymbol{\varepsilon}_0}^2) \ge 
        \frac{3\eta}{32} g_{\text{thres}}^2\ge \frac{3c^2\epsilon^2}{32\ell\chi^4}
        \end{equation*}
        
        \item $\norm*{\mathbf{z}_0}=\norm*{ q\left(\mathbf{x}_0,\frac{g_{\text{thres}}}{4c_h}\right)}\le \frac{3}{4}g_{\text{thres}}$\\
        In this case, $\norm*{\nabla f(\mathbf{x}_0)}\le g_{\text{thres}}$ and Algorithm \ref{algo:PAGD} will add a perturbation $\boldsymbol{\xi}$ of radius $r$ such that $\tilde{\mathbf{x}}_{0}\leftarrow \mathbf{x}_0+\boldsymbol{\xi}$, and will perform approximate gradient descent (without perturbations) for at most $t_{\text{thres}}$ steps. Since $\mathbf{x}_0$ is not a second-order stationary point then by Lemma \ref{lem:main_saddle_restated} there exists $i^\prime\le t_{\text{thres}}$  such that:
        
        \begin{equation*}
         f(\mathbf{x}_0)-f(\mathbf{x}_1) =f(\mathbf{x}_0)-f(\tilde{\mathbf{x}}_{i^\prime})  \ge f_{\text{thres}} = \frac{c}{\chi^3}\cdot\sqrt{\frac{\epsilon^3}{\rho}}
        \end{equation*}
        This means on average every step decreases the function value by
        \begin{equation*}
        \frac{f(\mathbf{x}_0)-f(\tilde{\mathbf{x}}_{i^\prime})}{i^\prime} \ge \frac{f_{\text{thres}}}{t_{\text{thres}}} = \frac{c^3}{\chi^4}\cdot\frac{\epsilon^2}{\ell}
        \end{equation*}
        \end{enumerate}
        
        Hence, we can conclude that as long as Algorithm \ref{algo:PAGD} has not terminated yet, on average, every step decreases function value by at least $\frac{c^3}{\chi^4}\cdot\frac{\epsilon^2}{\ell}$. However, we clearly can not decrease function value by more than $f(\mathbf{x}_0) - f^\star$, where $f^\star$ is the minimum value of $f$. This means Algorithm \ref{algo:PAGD} must terminate within the following number of iterations:
        \begin{equation*}
        \frac{f(\mathbf{x}_0) - f^\star}{\frac{c^3}{\chi^4}\cdot\frac{\epsilon^2}{\ell}}
        =\frac{\chi^4}{c^3}\cdot \frac{\ell(f(\mathbf{x}_0) - f^\star)}{\epsilon^2} = O\left(\frac{\ell(f(\mathbf{x}_0) - f^\star)}{\epsilon^2}\log^{4}\left(\frac{d\ell\Delta_f}{\epsilon^2\delta}\right) \right)
        \end{equation*}
        
        ~
        
        Finally, we have to ensure that the above statement holds with high probability. In the worst case scenario, in each outer-loop iteration the algorithm will be enforced to add a perturbation yielding a decrease of $f_{\text{thres}}$.  Thus, the maximum number of perturbations are at most: 
        \begin{equation*}\frac{f(\mathbf{x}_0) - f^\star}{f_{\text{thres}}}=\frac{f(\mathbf{x}_0) - f^\star}{\frac{c}{\chi^3}\cdot\sqrt{\frac{\epsilon^3}{\rho}}}\end{equation*}
        
        Applying Lemma \ref{lem:main_saddle_restated}, we know that the guaranteed decrease of $f_{\text{thres}}$ happens with probability at least $1-\frac{d\ell}{\sqrt{\rho\epsilon}}e^{-\chi}$ each time. 
        By union bound, the probability that all perturbations satisfy the decrease guarantee is at least
        \begin{equation*}
        1- \frac{d\ell}{\sqrt{\rho\epsilon}}e^{-\chi} \cdot \frac{f(\mathbf{x}_0) - f^\star}{\frac{c}{\chi^3}\cdot\sqrt{\frac{\epsilon^3}{\rho}}}
        = 1 -   \frac{\chi^3e^{-\chi}}{c}\cdot  \frac{d\ell(f(\mathbf{x}_0) - f^\star)}{\epsilon^2}
        \end{equation*}
        
        Recall our choice of $\chi = 3\max\{\log(\frac{d\ell\Delta_f}{c\epsilon^2\delta}), 4\}$. Since $\chi \ge 12$, we have $\chi^3e^{-\chi} \le e^{-\chi/3}$, this gives:
        \begin{equation*}
        \frac{\chi^3e^{-\chi}}{c}\cdot  \frac{d\ell(f(\mathbf{x}_0) - f^\star)}{\epsilon^2} 
        \le e^{-\chi/3}  \frac{d\ell(f(\mathbf{x}_0) - f^\star)}{c\epsilon^2} \le \delta
        \end{equation*}
        which finishes the proof.
        
    \end{proof}
    
    What remains to be proven is why adding a perturbation is guaranteed to help the algorithm decrease the value of $f$ substantially with high probability. 
    Following the proof strategy of \cite{jin2017escape} we will define some additional notation.  Let the condition number be the ratio of the Lipschitz constant of  $\nabla f$  and the smallest negative eigenvalue of the Hessian of $\mathbf{x}_t$ before adding the perturbation, i.e $\kappa = \ell / \gamma \ge 1$. Additionally we define the following units:
        \[
            p\gets \log (\tfrac{d\kappa}{\delta}),
            \mathfrak{L} \gets \eta\ell,
            \mathfrak{F} \gets \frac{\mathfrak{L}}{p^3}\frac{\gamma^3}{\rho^2},
            \mathfrak{G} \gets \frac{\sqrt{\mathfrak{L}}}{p^2} \frac{\gamma^2}{\rho} ,
            \mathcal{S} \gets \frac{\sqrt{\mathfrak{L}}}{p}\frac{\gamma}{\rho},
            \mathfrak{R} \gets \frac{2\mathcal{S}}{\kappa p} ,
            \mathfrak{T} \gets \frac{p}{\eta \gamma}
        \]
        Following the above definitions,  it holds that: $ \mathcal{S} = \sqrt{\tfrac{\mathfrak{F}p}{\gamma}}=\tfrac{\mathfrak{G}p}{\gamma}$, $\ell\mathfrak{R}=2\mathfrak{G}$ and $\eta \mathfrak{T}\mathfrak{G}=\mathcal{S}$

    \shadowbox{\begin{minipage}{6in}
    
    (A): The first argument in this proof is that if the $\tilde{\mathbf{x}}_i$ iterates do not achieve a decrease of $2.5\mathfrak{F}$ in $\mathfrak{c}\mathfrak{T}$ steps then they must remain confined in a small ball around  $\tilde{\mathbf{x}}_0$. 
    
    \end{minipage}} 
    \begin{lemma}\label{lemma:16}
For any constant $\mathfrak{c} \ge 3$, define: 
\begin{equation*}
T = \min\left\{ ~\inf_t\left\{t| f(\mathbf{u}_0) - f(\mathbf{u}_t)  \ge 2.5\mathfrak{F} \right\},  \mathfrak{c}\mathfrak{T}~\right\}
\end{equation*}
then, for any $\eta \le 1/\ell$, we have for all $t<T$ that $\norm{\mathbf{u}_t - \mathbf{u}_0} \le 100( \mathcal{S} \cdot \mathfrak{c} )$.
\end{lemma}
    \begin{proof}[Proof of Lemma \ref{lemma:16}]
        Applying repeatedly Lemma \ref{lemma:step-convergence}, we get for $t < T$
        \begin{equation*}
            f(\mathbf{u}_t) - f(\mathbf{u}_0) \leq -\frac{\eta}{2}\sum_{i=0}^t \left(\norm{\nabla f(\mathbf{u}_i)}^2-\norm{\boldsymbol{\varepsilon}_i}^2 \right)
        \end{equation*}
        where 
        \begin{equation*}
                    \boldsymbol{\varepsilon}_i = q_x(\mathbf{u}_i,h_{low}) - \nabla f(\mathbf{u}_i).
        \end{equation*}
        By definition of $T$ we have that the function value of $f$ has not yet decreased by $2.5\mathfrak{F}$.
        \begin{align*}
           \frac{\eta}{2}\sum_{i=0}^t \norm{\nabla f(\mathbf{u}_i)}^2 &\leq  f(\mathbf{u}_0) - f(\mathbf{u}_t) +  \frac{\eta}{2} \sum_{i=0}^t \norm{\boldsymbol{\varepsilon}_i}^2 \\
           \frac{\eta}{2}\sum_{i=0}^t \norm{\nabla f(\mathbf{u}_i)}^2 &\leq  2.5\mathfrak{F} +  \frac{\eta}{2} \sum_{i=0}^t \norm{\boldsymbol{\varepsilon}_i}^2
        \end{align*}
        Since $T \leq \mathfrak{c}\mathfrak{T}$ and also $\norm{\boldsymbol{\varepsilon}_i} \leq \mathfrak{G}$ we then have
        \begin{align*}
           \frac{\eta}{2}\sum_{i=0}^t \norm{\nabla f(\mathbf{u}_i)}^2 &\leq  2.5\mathfrak{F} +  \frac{\eta}{2}\mathfrak{G}^2\mathfrak{c}\mathfrak{T}\\
           \sum_{i=0}^t \norm{\nabla f(\mathbf{u}_i)}^2 &\leq  \frac{5}{\eta}\mathfrak{F} +  \mathfrak{G}^2\mathfrak{c}\mathfrak{T}\\
           \sum_{i=0}^t \left(\norm{\nabla f(\mathbf{u}_i)}^2 + \norm{\boldsymbol{\varepsilon}_i}^2 \right) &\leq  \frac{5}{\eta}\mathfrak{F} +  2\mathfrak{G}^2\mathfrak{c}\mathfrak{T}
        \end{align*}
        We also have that $\norm{q_x(\mathbf{u}_i,h_{low})}^2 \leq 2 \left( \norm{\nabla f(\mathbf{u}_i)}^2 + \norm{\boldsymbol{\varepsilon}_i}^2 \right)$. Therefore we have that
        \begin{equation*}
            \sum_{i=0}^t\norm{q_x(\mathbf{u}_i,h_{low})}^2 \leq \frac{10}{\eta}\mathfrak{F} +  4\mathfrak{G}^2\mathfrak{c}\mathfrak{T}
        \end{equation*}
       Now we can bound the difference between $\mathbf{u}_t$ and $\mathbf{u}_0$:
        \begin{align*}
          \norm{\mathbf{u}_t - \mathbf{u}_0}^2 &=  \norm*{\sum_{i=1}^t \mathbf{u}_i - \mathbf{u}_{i-1}}^2 \\
          &\leq t\sum_{i=1}^t \norm{\mathbf{u}_i - \mathbf{u}_{i-1}}^2 \\
          &\leq t \eta^2 \sum_{i=0}^t\norm{q_x(\mathbf{u}_i,h_{low})}^2 \\
          &\leq t \eta^2 \left(  \frac{10}{\eta}\mathfrak{F} +  4\mathfrak{G}^2\mathfrak{c}\mathfrak{T} \right)\\
          &\leq t \eta^2 \left(  \frac{10}{\eta}\mathfrak{F} +  4\mathfrak{G}^2\mathfrak{c}\mathfrak{T} \right)\\
          &\leq \mathfrak{c}\mathfrak{T}  \eta^2 \left(  \frac{10}{\eta}\mathfrak{F} +  4\mathfrak{G}^2\mathfrak{c}\mathfrak{T} \right)
        \end{align*}
        Manipulating the constants we get
        \begin{align*}
           \norm{\mathbf{u}_t - \mathbf{u}_0}^2 &\leq \left(10\mathfrak{c} + \mathfrak{c}^2 \right)  \mathcal{S}^2 \\
           \norm{\mathbf{u}_t - \mathbf{u}_0} &\leq \sqrt{\left(10\mathfrak{c} + \mathfrak{c}^2 \right)}  \mathcal{S}
        \end{align*}
        For any $\mathfrak{c} \ge 3$ we have
        \begin{equation*}
            \norm{\mathbf{u}_t - \mathbf{u}_0} \leq 100 (\mathfrak{c} \mathcal{S})
        \end{equation*}
    \end{proof}
\clearpage    
\shadowbox{\begin{minipage}{6in}
        (B):The second step in our proof strategy is to show that if all the iterates from $\mathbf{u}_0$ are constrained in a small ball, iterates from $\mathbf{w}_0=\mathbf{u}_0 +\mu \cdot \frac{\mathfrak{R}}{2}\mathbf{e}_1$, for large enough $\mu$ must be able to decrease the function value. In order to do that, we keep track of vector $\mathbf{v}$ which is the difference between $\{\mathbf{u}_i\}$ and $\{\mathbf{w}_i\}$. We also decompose $\mathbf{v}$ into two different eigenspaces: the direction $\mathbf{e}_1$ (the minimum-eigenvalue eigenvector) and its orthogonal subspace.
\end{minipage}}

    \begin{lemma}\label{lemma:17}
        There exists absolute constant $c_{\max}, \mathfrak{c}$ such that: for any $\delta\in (0, \frac{d\kappa}{e}]$, let $f(\cdot), \hat{\mathbf{x}}$ satisfies the following conditions
        \begin{equation*}
            \norm{\nabla f(\hat{\mathbf{x}})} \le \mathfrak{G} \quad \quad \text{~and~} \quad \quad \lambda_{\min}(\nabla^2 f(\hat{\mathbf{x}})) \le -\gamma
        \end{equation*}
        and any two sequences $\{\mathbf{u}_t\}, \{\mathbf{w}_t\}$  with initial points $\mathbf{u}_0, \mathbf{w}_0$ satisfying: 
        \begin{align*}
            \mathbf{w}_0 = \mathbf{u}_0 +\mu \cdot \frac{\mathfrak{R}}{2} \cdot \mathbf{e}_1 ,\quad  \mu \in [\delta/(2\sqrt{d}), 1], \quad \norm{\mathbf{u}_0 - \hat{\mathbf{x}}} \leq \frac{\mathfrak{R}}{2}
        \end{align*}
        $\mathbf{e}_1$ is the eignevector of the minimum eigenvalue of $\nabla^2 f(\hat{\mathbf{x}})$. Assume also that $h_{low} \leq \frac{\rho\mathcal{S}\delta}{2c_h\sqrt{d}}\frac{\mathfrak{R}}{2}$. Define 
        \begin{equation*}
        T = \min\left\{ ~\inf_t\left\{t| f(\mathbf{w}_0) - f(\mathbf{w}_t)  \ge 2.5\mathfrak{F} \right\},  \mathfrak{c}\mathfrak{T}~\right\}
        \end{equation*}
        then, for any $\eta \le c_{\max} / \ell$, if $\norm{\mathbf{u}_t - \mathbf{u}_0} \le 100 (\mathcal{S}\cdot \mathfrak{c} )$ for all $t<T$, we will have $T < \mathfrak{c} \mathfrak{T}$.
    \end{lemma}

    \begin{proof}[Proof of Lemma \ref{lemma:17}]
    Recall notation $\tilde{H} = \nabla^2 f(\hat{\mathbf{x}})$. Since $\delta \in (0, \frac{d\kappa}{e}]$, we always have $p\ge 1$. Define $\mathbf{v}_t = \mathbf{w}_t - \mathbf{u}_t$, by assumption, we have $\mathbf{v}_0 = \mu \tfrac{\mathfrak{R}}{2} \mathbf{e}_1$. Let us firstly define the gradient approximation errors for these two sequences
    \begin{align*}
        \boldsymbol{\varepsilon}_{\mathbf{w}_t} &= q_x(\mathbf{w}_t,h_{low}) - \nabla f(\mathbf{w}_t) \\
        \boldsymbol{\varepsilon}_{\mathbf{u}_t} &= q_x(\mathbf{u}_t,h_{low}) - \nabla f(\mathbf{u}_t) \\
    \end{align*}
    Now, consider the update equation for $\mathbf{w}_t$:
    \begin{align*}
        \mathbf{u}_{t+1} + \mathbf{v}_{t+1} 
        =& \mathbf{w}_{t+1} \\  
        =&\mathbf{w}_t - \eta q_x(\mathbf{w}_t,h_{low}) \\
        =&\mathbf{w}_t - \eta (\nabla f(\mathbf{w}_t) + \boldsymbol{\varepsilon}_{\mathbf{w}_t} )\\
        =&\mathbf{u}_t + \mathbf{v}_t - \eta \nabla f(\mathbf{u}_t + \mathbf{v}_t) - \eta \boldsymbol{\varepsilon}_{\mathbf{w}_t}\\
        =&\mathbf{u}_t + \mathbf{v}_t - \eta \nabla f(\mathbf{u}_t) - \eta \left[\int_{0}^1 \nabla^2 f(\mathbf{u}_t + \theta \mathbf{v}_t) \mathrm{d}\theta\right] \mathbf{v}_t -\eta \boldsymbol{\varepsilon}_{\mathbf{w}_t} \\ 
        =&\mathbf{u}_t + \mathbf{v}_t - \eta \nabla f(\mathbf{u}_t) - \eta(\tilde{H} + \Delta'_t) \mathbf{v}_t -\eta \boldsymbol{\varepsilon}_{\mathbf{w}_t}\\
        =&\mathbf{u}_t - \eta \nabla f(\mathbf{u}_t) + (I - \eta \tilde{H} - \eta \Delta'_t) \mathbf{v}_t -\eta \boldsymbol{\varepsilon}_{\mathbf{w}_t}\\
        =&\mathbf{u}_t - \eta (\nabla f(\mathbf{u}_t) + \boldsymbol{\varepsilon}_{\mathbf{u}_t}) + (I - \eta \tilde{H} - \eta \Delta'_t) \mathbf{v}_t  -\eta (\boldsymbol{\varepsilon}_{\mathbf{w}_t} - \boldsymbol{\varepsilon}_{\mathbf{u}_t})\\
        =&\mathbf{u}_t - \eta q_x(\mathbf{u}_t,h_{low}) + (I - \eta \tilde{H} - \eta \Delta'_t) \mathbf{v}_t  -\eta (\boldsymbol{\varepsilon}_{\mathbf{w}_t} - \boldsymbol{\varepsilon}_{\mathbf{u}_t})\\
        =&\mathbf{u}_{t+1}  + (I - \eta \tilde{H} - \eta \Delta'_t) \mathbf{v}_t  -\eta (\boldsymbol{\varepsilon}_{\mathbf{w}_t} - \boldsymbol{\varepsilon}_{\mathbf{u}_t})
    \end{align*}
    where 
    \begin{equation*}
        \Delta'_t = \int_{0}^1 \nabla^2 f(u_t + \theta v_t) \mathrm{d}\theta - \tilde{H}
    \end{equation*}
    
    This gives the dynamic for $\mathbf{v}_t$ satisfy:
    \begin{equation}\label{eq:v_dynamic}
    \mathbf{v}_{t+1} = (I - \eta \tilde{H} - \eta \Delta'_t) \mathbf{v}_t  -\eta (\boldsymbol{\varepsilon}_{\mathbf{w}_t} - \boldsymbol{\varepsilon}_{\mathbf{u}_t})
    \end{equation}
    
    Since $f$ is Hessian Lipschitz, we have
    \begin{equation*}
        \norm{\Delta'_t} = \norm*{\int_{0}^1 \nabla^2 f(\mathbf{u}_t + \theta \mathbf{v}_t) - \nabla^2f(\hat{\mathbf{x}}) \mathrm{d}\theta
    } \le \int_{0}^{1} \rho \norm{\mathbf{u}_t+\theta \mathbf{v}_t - \hat{\mathbf{x}}} \mathrm{d}\theta \le \rho( \norm{\mathbf{u}_t- \mathbf{u}_0} + \norm{\mathbf{v}_t}+ \norm{\hat{\mathbf{x}} - \mathbf{u}_0}).
    \end{equation*}
    For $t< T$ the sequence $\{\mathbf{w}_t\}$ has not decreased the function $f$ by $-2.5\mathfrak{F}$. In other words, it holds that  $ f(\mathbf{w}_0) - f(\mathbf{w}_t)  \le 2.5\mathfrak{F}$, so applying Lemma \ref{lemma:16}, we know for all $t \le T$
    \begin{equation*}
     \norm{\mathbf{w}_t-\mathbf{w}_0}\le 100( \mathcal{S}\mathfrak{c} ). 
    \end{equation*}
    By condition of Lemma \ref{lemma:17}, we know $\norm{\mathbf{u}_t- \mathbf{u}_0} \le 100( \mathcal{S}\mathfrak{c} )$ for all $t<T$.
    This gives for all $t<T$:
    \begin{align} \label{eq:bound_v}
    \norm{\mathbf{v}_t} = \norm{\mathbf{w}_t - \mathbf{u}_t} &= \norm{(\mathbf{w}_t - \mathbf{w}_0) - (\mathbf{u}_t - \mathbf{u}_0) + (\mathbf{w}_0 - \mathbf{u}_0 ) }  \nonumber\\
     &\le \norm{(\mathbf{w}_t - \mathbf{w}_0)} + \norm{\mathbf{u}_t - \mathbf{u}_0} + \norm{\mathbf{w}_0 - \mathbf{u}_0} \nonumber\\
     &\le 100( \mathcal{S}\mathfrak{c} ) + 100( \mathcal{S}\mathfrak{c} ) + \mu\frac{\mathfrak{R}}{2} \nonumber\\
     &\le 200( \mathcal{S}\mathfrak{c} ) + \frac{\mathfrak{R}}{2} \nonumber\\
     &\le (200\mathfrak{c} +1 )\mathcal{S}
    \end{align}
    where the last step holds because $\frac{\mathfrak{R}}{2} \leq \mathcal{S}$
    This gives us for $t<T$:
    \begin{equation*}
    \norm{\Delta'_t} \le \rho( \norm{\mathbf{u}_t- \mathbf{u}_0} + \norm{\mathbf{v}_t}+ \norm{\hat{\mathbf{x}} - \mathbf{u}_0})
    \le \rho( 100\mathfrak{c} \mathcal{S}+(200\mathfrak{c} +1 )\mathcal{S}+\tfrac{\mathfrak{R}}{2})
    \le \rho\mathcal{S} (300 \mathfrak{c} + 2)
    \end{equation*}

    Let $\psi_t$ be the norm of $\mathbf{v}_t$ projected onto $\mathbf{e}_1$ direction and the normal vector and $\varphi_t$ correspondingly be the norm of $\mathbf{v}_t$ projected onto remaining subspace. Let us define as $\lambda = \eta\rho\mathcal{S} (300 \mathfrak{c} + 2)$. Equation \ref{eq:v_dynamic} gives us:
    \begin{align*}
    \psi_{t+1} &= \norm*{\prod_{\mathbf{e}_1}(I - \eta \tilde{H})\mathbf{v}_t - \eta \Delta'_t \mathbf{v}_t-\eta (\boldsymbol{\varepsilon}_{\mathbf{w}_t} - \boldsymbol{\varepsilon}_{\mathbf{u}_t})}\\
    \varphi_{t+1} &= \norm*{\prod_{\mathbb{R}^d\setminus\{\mathbf{e}_1\}}(I - \eta \tilde{H})\mathbf{v}_t - \eta \Delta'_t \mathbf{v}_t-\eta (\boldsymbol{\varepsilon}_{\mathbf{w}_t} - \boldsymbol{\varepsilon}_{\mathbf{u}_t})}\\
    \end{align*}
    Lower bound of $\psi_{t+1}$:
    \begin{align*}
    \psi_{t+1} &= \norm*{\prod_{\mathbf{e}_1}[(I - \eta \tilde{H})\psi_t \mathbf{e}_1 - \eta \Delta'_t \mathbf{v}_t-\eta(\boldsymbol{\varepsilon}_{\mathbf{w}_t} - \boldsymbol{\varepsilon}_{\mathbf{u}_t})]}\\
    &\ge \norm{(I - \eta \tilde{H})\psi_t \mathbf{e}_1}  - \eta\norm{ \prod_{\mathbf{e}_1}[ \Delta'_t \mathbf{v}_t]}-\eta\norm{ \prod_{\mathbf{e}_1}[\boldsymbol{\varepsilon}_{\mathbf{w}_t} - \boldsymbol{\varepsilon}_{\mathbf{u}_t}]}\\
    &\ge (1+\gamma \eta)\psi_t - \eta \norm{\Delta'_t \mathbf{v}_t}-\eta\norm{ \boldsymbol{\varepsilon}_{\mathbf{w}_t} - \boldsymbol{\varepsilon}_{\mathbf{u}_t}}\\
    &\ge (1+\gamma \eta)\psi_t - \eta\norm{\Delta'_t}\norm{\mathbf{v}_t}-\eta\norm{ \boldsymbol{\varepsilon}_{\mathbf{w}_t} - \boldsymbol{\varepsilon}_{\mathbf{u}_t}}\\
    &\ge (1+\gamma \eta)\psi_t - \lambda\sqrt{\psi_t^2 + \varphi_t^2}-\eta\norm{ \boldsymbol{\varepsilon}_{\mathbf{w}_t} - \boldsymbol{\varepsilon}_{\mathbf{u}_t}}\\
    \end{align*}
    Upper bound of $\varphi_{t+1}$:
    \begin{align*}
    \varphi_{t+1} &= \norm{\prod_{\mathbb{R}^d\setminus\{\mathbf{e}_1\}}[(I - \eta \tilde{H})\mathbf{v}_t - \eta \Delta'_t \mathbf{v}_t-\eta(\boldsymbol{\varepsilon}_{\mathbf{w}_t} - \boldsymbol{\varepsilon}_{\mathbf{u}_t})]}\\
    &\le \norm{\prod_{\mathbb{R}^d\setminus\{\mathbf{e}_1\}}[(I - \eta \tilde{H})\mathbf{v}_t]}+
    \norm{ \prod_{\mathbb{R}^d\setminus\{\mathbf{e}_1\}}[\eta \Delta'_t \mathbf{v}_t]}+\eta\norm{\prod_{\mathbb{R}^d\setminus\{\mathbf{e}_1\}}[\boldsymbol{\varepsilon}_{\mathbf{w}_t} - \boldsymbol{\varepsilon}_{\mathbf{u}_t}]}\\
    \\
    &\le \norm{\prod_{\mathbb{R}^d\setminus\{\mathbf{e}_1\}}[(I - \eta \tilde{H})\mathbf{v}_t]}+
    \norm{ \eta \Delta'_t \mathbf{v}_t}+\eta\norm{\boldsymbol{\varepsilon}_{\mathbf{w}_t} - \boldsymbol{\varepsilon}_{\mathbf{u}_t}}\\
    \\
    &\le (1+\gamma \eta)\varphi_t + \lambda\sqrt{\psi_t^2 + \varphi_t^2}+
    \eta\norm{\boldsymbol{\varepsilon}_{\mathbf{w}_t} - \boldsymbol{\varepsilon}_{\mathbf{u}_t}}\\
    \end{align*}
    Therefore we have
    \begin{align*}
    \psi_{t+1} \ge& (1+\gamma \eta)\psi_t - \lambda\sqrt{\psi_t^2 + \varphi_t^2}-\eta\norm{\boldsymbol{\varepsilon}_{\mathbf{w}_t} - \boldsymbol{\varepsilon}_{\mathbf{u}_t}}\\
    \varphi_{t+1} \le &(1+\gamma\eta)\varphi_t + \lambda\sqrt{\psi_t^2 + \varphi_t^2}+\eta\norm{\boldsymbol{\varepsilon}_{\mathbf{w}_t} - \boldsymbol{\varepsilon}_{\mathbf{u}_t}}
    \end{align*}
    We will now prove via induction the following fact:
    \begin{claim} \label{bound-fact} $\forall t<T \ \ 
    \varphi_t \le 4 \lambda t \cdot \psi_t$ and $\norm{\boldsymbol{\varepsilon}_{\mathbf{w}_t}}\le \frac{\lambda}{2 \eta}\norm{\mathbf{v}_t}$ and $\norm{\boldsymbol{\varepsilon}_{\mathbf{u}_t}}\le \frac{\lambda}{2 \eta}\norm{\mathbf{v}_t}$
    \end{claim}
    \begin{proof}
    Let us prove the base case of the induction:
    \begin{itemize}
        \item By hypothesis of Lemma \ref{lemma:17}, we know $\varphi_0 = 0$ so $\varphi_0 \le 4 \lambda 0 \cdot \psi_0$ holds trivially
        \item Based on the choice of $h_{low}$ we have that 
        \begin{align*}
            \norm{\boldsymbol{\varepsilon}_{\mathbf{w}_t}} \leq \rho\mathcal{S}\dfrac{\delta}{2\sqrt{d}} \dfrac{\mathfrak{R}}{2}\leq \frac{\lambda}{2 \eta} \psi_0 \leq \frac{\lambda}{2 \eta}\norm{\mathbf{v}_0}\\
            \norm{\boldsymbol{\varepsilon}_{\mathbf{u}_t}} \leq \rho\mathcal{S}\dfrac{\delta}{2\sqrt{d}} \dfrac{\mathfrak{R}}{2}\leq \frac{\lambda}{2 \eta} \psi_0 \leq \frac{\lambda}{2 \eta}\norm{\mathbf{v}_0}.
        \end{align*}
    \end{itemize}
    Thus the base case of induction holds. Assume Claim \ref{bound-fact} is true for $\tau\le t$. Now we can rewrite the inequalities based on the inductive hypothesis as follows:
    \begin{align*}
    \psi_{t+1} \ge& (1+\gamma \eta)\psi_t - 2\lambda\sqrt{\psi_t^2 + \varphi_t^2}\\
    \varphi_{t+1} \le &(1+\gamma\eta)\varphi_t + 2\lambda\sqrt{\psi_t^2 + \varphi_t^2}
    \end{align*}
    For $t+1 \le T$, we have:
    \begin{equation*}
    \begin{Bmatrix}4\lambda(t+1)\psi_{t+1} 
    &\ge & 4\lambda (t+1) \left( (1+\gamma \eta)\psi_t - 2\lambda \sqrt{\psi_t^2 + \varphi_t^2}\right)\\
    \varphi_{t+1} &\le &4 \lambda  t(1+\gamma\eta) \psi_t + 2\lambda \sqrt{\psi_t^2 + \varphi_t^2}
    \end{Bmatrix}
    \end{equation*}
    Thus it suffices to prove that:
    \begin{align*}
        4 \lambda  t(1+\gamma\eta) \psi_t + 2\lambda \sqrt{\psi_t^2 + \varphi_t^2} &\le 
        4\lambda (t+1) \left( (1+\gamma \eta)\psi_t - 2\lambda \sqrt{\psi_t^2 + \varphi_t^2} \right)\\
        \left(2+8\lambda (t+1)\right)\sqrt{\psi_t^2 + \varphi_t^2} &\le 4 (1+\gamma \eta)\psi_t .
    \end{align*}
    By choosing $\sqrt{c_{\max}}\le \frac{1}{300\mathfrak{c}+2}\min\{\frac{1}{2\sqrt{2}}, \frac{1/3}{8\mathfrak{c}}\}$, using the facts 
    $\begin{cases}\eta\rho \mathcal{S}\mathfrak{T}=\sqrt{\eta\ell}\\
    \eta \le c_{\max}/\ell\end{cases}$, we have 
    \begin{equation*}
        8\lambda (t+1) \le 8\lambda T \le 
    8\eta\rho  \mathcal{S}(300 \mathfrak{c} + 2)\mathfrak{c}\mathfrak{T} =8\sqrt{\eta\ell}(300 \mathfrak{c} + 2)\mathfrak{c}\le 1/3
    \end{equation*}
    This gives:
    \begin{equation*}
        4 (1+\gamma \eta)\psi_t \ge 4\psi_t \ge \tfrac{7}{3}\sqrt{2\psi_t^2}\ge \left(2+8\lambda (t+1)\right)\sqrt{\psi_t^2 + \varphi_t^2}
    \end{equation*} which finishes the induction of the first part.

    Now, using again the induction hypothesis, we know $\varphi_t \le 4  \lambda t \cdot \psi_t \le \psi_t$, this gives:
    \begin{equation}
    \psi_{t+1} \ge (1+\gamma \eta)\psi_t - \sqrt{2}\lambda\psi_t
    \ge (1+\frac{\gamma \eta}{2})\psi_t \label{eq:growth_v}
    \end{equation}
    where the last step follows from 
    \begin{equation*}
       \sqrt{2} \lambda= \sqrt{2}\eta \rho \mathcal{S}(300 \mathfrak{c} + 2)=\sqrt{2}\sqrt{\eta \ell}\tfrac{\gamma}{\rho p} \le \sqrt{c_{\max}}(300 \mathfrak{c} + 2) \gamma \frac{\eta}{p} < \frac{\gamma \eta}{2}.
    \end{equation*}
    Equation \ref{eq:growth_v} yields that $\psi_t$ is increasing sequence. Clearly
    \begin{align*}
        \norm{\boldsymbol{\varepsilon}_{\mathbf{w}_{t+1}}} \leq \frac{\lambda}{2\eta} \psi_0 \leq \frac{\lambda}{2\eta} \psi_{t+1} \leq \frac{\lambda}{2\eta} \norm{\mathbf{v}_{t+1}}\\
        \norm{\boldsymbol{\varepsilon}_{\mathbf{u}_{t+1}}} \leq \frac{\lambda}{2\eta} \psi_0 \leq \frac{\lambda}{2\eta} \psi_{t+1} \leq \frac{\lambda}{2\eta} \norm{\mathbf{v}_{t+1}}
    \end{align*}
    Thus we have completed the induction.
    \end{proof}
    Finally, combining Eq.\eqref{eq:bound_v} and \eqref{eq:growth_v} we have for all $t<T$:
    \begin{equation*}
    (200\mathfrak{c}+1)\mathcal{S}
    \ge \norm{\mathbf{v}_t} \ge \psi_t \ge (1+\frac{\gamma \eta}{2})^t \psi_0
    = (1+\frac{\gamma \eta}{2})^t \mu\tfrac{\mathfrak{R}}{2}=  (1+\frac{\gamma \eta}{2})^t \frac{\mathcal{S}}{\kappa}\frac{1}{p}
    = (1+\frac{\gamma \eta}{2})^t \frac{\delta}{2\sqrt{d}}\frac{\mathcal{S}}{\kappa}\frac{1}{p}
    \end{equation*}
    This implies:
    \begin{equation*}
    T < \frac{\log (\frac{(200\mathfrak{c}+1)}{2\sqrt{d}} \frac{\kappa d}{\delta}\cdot p)}{\log (1+\frac{\gamma \eta}{2})}
    \le \frac{\log ((200\mathfrak{c}+1)) +  \log(\frac{\kappa d} {\delta}) +\log p}{ (\frac{\gamma \eta}{2})}
    \le \frac{2\log (200\mathfrak{c}+1)}{\gamma \eta}+2\frac{\log(\frac{\kappa d} {\delta})}{\gamma\eta}+2\frac{p}{\gamma\eta}
    \end{equation*}
    The last inequality is due to the following facts 
    \begin{itemize}
        \item $p=\log(\frac{\kappa d} {\delta}) \ge 1$ and $\forall x\ge 1 :\log x\le x$.
        \item $\forall x\ge 0:  \log (1+x) \le x$ thus $\log (1+\frac{\gamma \eta}{2}) \le \frac{\gamma \eta}{2}$.
        \item $\mathfrak{T}=\frac{p}{\gamma\eta}$
    \end{itemize}
    Therefore, it holds that:
    \begin{equation*}
        T < 2\log (200\mathfrak{c}+1)\frac{p}{\gamma\eta} + 4\mathfrak{T}\le \mathfrak{T}(2 \log (200\mathfrak{c}+1) +4)
    \end{equation*}
    By choosing constant $\mathfrak{c}$ to be large enough to satisfy $2 \log (200\mathfrak{c}+1) +4) \le \mathfrak{c}$,  for example (i.e $ \mathfrak{c}\ge 21$), we will have $T < \mathfrak{c}\mathfrak{T} $, which finishes the proof. 
\end{proof}
\clearpage
\shadowbox{\begin{minipage}{15cm}
   (C): Until now we have proved that firstly if approximate gradient descent from $\mathbf{u}_0$ does not decrease function value, then all the iterates must lie within a small ball around $\mathbf{u}_0$ (Lemma \ref{lemma:16}) and secondly starting an approximate descent from $\mathbf{w}_0$, which is $\mathbf{u}_0$ but displaced along $\mathbf{e}_1$ direction (negative eigenvalue's eigenvector for at least a certain distance), will decreases the function value if $\{\mathbf{u}_t\}$ is bounded. (Lemma \ref{lemma:17}). 
 \end{minipage}}

 The following lemma combines the above two lemmas:
\begin{lemma}\label{lemma:15} 
There exists a universal constant $\hat{c}_{\max}$, for any $\delta\in (0, \frac{d\kappa}{e}]$, let $f(\cdot), \hat{\mathbf{x}}$ satisfies the following conditions 
\begin{equation*}
    \norm{\nabla f(\hat{\mathbf{x}})} \le \mathfrak{G} \quad \quad \text{~and~} \quad \quad \lambda_{\min}(\nabla^2 f(\hat{\mathbf{x}})) \le -\gamma
\end{equation*}
and $\mathbf{e}_1$ be the minimum eigenvector of $\nabla^2 f(\hat{\mathbf{x}})$. Consider two algorithm sequences $\{\mathbf{u}_t\}, \{\mathbf{w}_t\}$ with initial points $\mathbf{u}_0, \mathbf{w}_0$ satisfying: 
\begin{equation*}
\norm{\mathbf{u}_0 - \hat{\mathbf{x}}} \le \tfrac{\mathfrak{R}}{2}, \quad \mathbf{w}_0 = \mathbf{u}_0 +\mu \cdot \tfrac{\mathfrak{R}}{2} \cdot \mathbf{e}_1, \quad\mu \in [\delta/(2\sqrt{d}), 1]
\end{equation*}
Then, for any step size $\eta \le \hat{c}_{\max} / \ell$,  at least one of the following is true 
\begin{itemize}
    \item there exists $T_u \le \frac{1}{\hat{c}_{\max}}\mathfrak{T}$ such that $f(\mathbf{u}_0)  - f(\mathbf{u}_{T_u}) \ge 2.5\mathfrak{F}$
    \item there exists $T_w \le \frac{1}{\hat{c}_{\max}}\mathfrak{T}$ such that $f(\mathbf{w}_{0})  - f(\mathbf{w}_{T_w}) \ge 2.5\mathfrak{F}$
\end{itemize}
\end{lemma}

\begin{proof}[Proof of Lemma \ref{lemma:15}]
Let $(c_{\max}^{(1)}, \mathfrak{c})$ be the absolute constant so that Lemma \ref{lemma:17} holds.  Choose 
\begin{equation*}
    \hat{c}_{\max}=\min\{1, c_{\max}^{(1)} ,\frac{1}{\mathfrak{c}}\}
\end{equation*}
Let $T^\star = \mathfrak{c}\mathfrak{T}$. Notice that by definition
$T^\star\le \frac{1}{\hat{c}_{\max}}\mathfrak{T} $. Finally , define:
\begin{equation*}
T^{\circ} = \inf_t\left\{t| f(\mathbf{u}_0) - f(\mathbf{u}_t)  \ge 2.5\mathfrak{F} \right\}
\end{equation*}
Let's consider following two cases:

\paragraph{Case $T^{\circ} \le T^\star$:} Clearly for this case we have for $T_u = T^{\circ}$ that
\begin{equation*}
    f(\mathbf{u}_0)  - f(\mathbf{u}_{T_u}) \ge 2.5\mathfrak{F}
\end{equation*}

\paragraph{Case $T^{\circ} > T^\star$:} In this case, by Lemma \ref{lemma:16}, we know $\norm{u_t-u_0}\le O(\mathcal{S} )$ for all $t\le T^\star$. Define 
\begin{equation*}
T^{\circ\circ} = \inf_t\left\{t| f(\mathbf{w}_0) - f(\mathbf{w}_t)  \ge 2.5\mathfrak{F} \right\}
\end{equation*}

By Lemma \ref{lemma:17}, we immediately have $T^{\circ\circ} \le T^\star=\mathfrak{c}\mathfrak{T}$. Clearly for this case we have for $T_u = T^{\circ\circ}$ we have that 
\begin{equation*}
    f(\mathbf{w}_0)  - f(\mathbf{w}_{T_w}) \ge 2.5\mathfrak{F}.
\end{equation*}
\end{proof}

\clearpage
\clearpage
\shadowbox{\begin{minipage}{15cm}
Having expanded the basic lemmas  ((A),(B),(C)) of \cite{jin2017escape} for the zero order case, we are able to use the basic geometric upper bound of the stuck region. For the sake of completeness we state again the main lemma:
 \end{minipage}}

\begin{lemma}\label{lem:main_saddle}
Let $f$ be a $\ell$-gradient Lipschitz and $\rho$-Hessian Lipschitz function. There exists universal constant $c_{\max}$, for any $\delta\in (0, \frac{d\kappa}{e}]$, suppose we start with point $\hat{\mathbf{x}}$ satisfying following conditions:
\begin{equation*}
\norm{\nabla f(\hat{\mathbf{x}})} \le \mathfrak{G} \quad \quad \text{~and~} \quad \quad \lambda_{\min}(\nabla^2 f(\hat{\mathbf{x}})) \le -\gamma
\end{equation*}
Let $\mathbf{x}_0 = \hat{\mathbf{x}} + \boldsymbol{\xi}$ where $\boldsymbol{\xi}$ come from the uniform distribution over ball with radius $r = \frac{R}{2}$,
and let $\mathbf{x}_t$ be the iterates of approximate gradient descent from $\mathbf{x}_0$ and $T=\frac{\mathfrak{T}}{c_{\max}}$. Then, when step size
$\eta \le c_{\max} / \ell$, with at least probability $1-\delta$, we have that:
\begin{equation*}
 \exists t \leq T: f(\hat{\mathbf{x}})-f(\mathbf{x}_t)  \ge \mathfrak{F}
\end{equation*}
\end{lemma}

\begin{proof}[Proof of Lemma \ref{lem:main_saddle}]
By adding perturbation, in worst case we increase function value by:
\begin{equation*}
f(\mathbf{x}_0) - f(\hat{\mathbf{x}}) \le \nabla f(\hat{\mathbf{x}})^\top\boldsymbol{\xi} +  \frac{\ell}{2} \norm{\boldsymbol{\xi}}^2 \le
\frac{3\ell}{8}\mathfrak{R}^2=
\frac{3\ell}{8}\frac{4\mathcal{S}^2}{\kappa^2p^2}=\frac{3\ell}{2}\frac{\frac{\mathfrak{F}p}{\gamma}}{\kappa^2p^2}
\le \frac{3}{2}\mathfrak{F}\frac{1}{\kappa p}
\le \frac{3}{2}\mathfrak{F}
\end{equation*}
 We know $\mathbf{x}_0$ come from the uniform distribution over $\mathcal{B}_{\hat{\mathbf{x}}}(r)$. Let $\mathcal{A} \subset \mathcal{B}_{\hat{\mathbf{x}}}(r)$ denote the set of bad starting points 
 \begin{equation*}
     \mathcal{A} = \{\mathbf{x} \in \mathcal{B}_{\hat{\mathbf{x}}}(r) | \quad \forall t \leq T : \quad f(\mathbf{x}_0)-f(\mathbf{x}_t) < 2.5\mathfrak{F}   \}
 \end{equation*}
 otherwise if $\mathbf{x}_0 \in B_{\hat{\mathbf{x}}}(r) \setminus \mathcal{A}$, we have that 
 \begin{equation*}
      \exists t \leq T : f(\mathbf{x}_0) - f(\mathbf{x}_t) \ge 2.5\mathfrak{F}
 \end{equation*}
By applying Lemma \ref{lemma:17}, we know for any $\mathbf{x}_0\in \mathcal{A}$, it is guaranteed that 
\begin{equation*}
    \mathbf{x}_0 \pm \mu r \mathbf{e}_1 \not \in \mathcal{A} \text{ where } \mu \in [\frac{\delta}{2\sqrt{d}}, 1]
\end{equation*}
where $\mathbf{e}_1$ is the eigenvector of $\nabla^2 f(\hat{\mathbf{x}})$ with the smallest negative eigenvalue.

Let us denote $I_{\mathcal{A}}(\cdot)$ be the indicator function of being inside set $\mathcal{A}$. For a vector $\mathbf{x}$  let us define the following quantities
\begin{align*}
    x_{\mathbf{e}_1} &= \langle \mathbf{x}, \mathbf{e}_1 \rangle\\
    \mathbf{x}_{\neg\mathbf{e}_1} &= \prod_{\mathbb{R}^d \setminus \{\mathbf{e}_1\}} \mathbf{x}
\end{align*}
Recall $\mathcal{B}^{(d)}(r)$ be $d$-dimensional ball with radius $r$. By calculus, this gives an upper bound on the volume of $\mathcal{A}$:
\begin{align*}
\text{Vol}(\mathcal{A}) =& \int_{\mathcal{B}^{(d)}_{\hat{\mathbf{x}}}(r)}  \mathrm{d}\mathbf{x} \cdot I_{\mathcal{A}}(\mathbf{x})\\
= & \int_{\mathcal{B}^{(d-1)}_{\hat{\mathbf{x}}}(r)} \mathrm{d} \mathbf{x}_{\neg\mathbf{e}_1} 
\int_{\hat{\mathbf{x}}_{\mathbf{e}_1} - \sqrt{r^2 - \norm{\hat{\mathbf{x}}_{\neg\mathbf{e}_1} - \mathbf{x}_{\neg\mathbf{e}_1}}^2}}^{{\hat{\mathbf{x}}_{\mathbf{e}_1} + \sqrt{r^2 - \norm{\hat{\mathbf{x}}_{\neg\mathbf{e}_1} - \mathbf{x}_{\neg\mathbf{e}_1}}^2}}}  \mathrm{d} x_{\mathbf{e}_1} \cdot  I_{\mathcal{A}}(\mathbf{x})\\
\le & \int_{\mathcal{B}^{(d-1)}_{\hat{\mathbf{x}}}(r)} \mathrm{d} \mathbf{x}_{\neg\mathbf{e}_1} 
\cdot\left(2\cdot \frac{\delta}{2\sqrt{d}}r \right) = \text{Vol}(\mathcal{B}_0^{(d-1)}(r))\times \frac{\delta r}{\sqrt{d}}
\end{align*}
Then, we immediately have the ratio:
\begin{align*}
\frac{\text{Vol}(\mathcal{A})}{\text{Vol}(\mathcal{B}^{(d)}_{\hat{\mathbf{x}}}(r))}
\le \frac{\frac{\delta r}{\sqrt{d}} \times \text{Vol}(\mathcal{B}^{(d-1)}_0(r))}{\text{Vol} (\mathcal{B}^{(d)}_0(r))}
= \frac{\delta}{\sqrt{\pi d}}\frac{\Gamma(\frac{d}{2}+1)}{\Gamma(\frac{d}{2}+\frac{1}{2})}
\le \frac{\delta}{\sqrt{\pi d}} \cdot \sqrt{\frac{d}{2}+\frac{1}{2}} \le \delta
\end{align*}
The second last inequality is by the property of Gamma function that $\frac{\Gamma(x+1)}{\Gamma(x+1/2)}<\sqrt{x+\frac{1}{2}}$ as long as $x\ge 0$.
Therefore, with at least probability $1-\delta$, $\mathbf{x}_0 \not \in \mathcal{A}$. In this case, we have that there exists a $t \leq T$:
\begin{align*}
f(\hat{\mathbf{x}}) - f(\mathbf{x}_t)   =& f(\hat{\mathbf{x}})  - f(\mathbf{x}_0) +  f(\mathbf{x}_0) -  f(\mathbf{x}_t)\\
\le & 2.5\mathfrak{F} - 1.5\mathfrak{F} \ge \mathfrak{F}
\end{align*}
which finishes the proof.
\end{proof}
    
It is easy to check that our initial Lemma \ref{lem:main_saddle_restated} can be derived by substituting 
$\eta =\frac{c}{\ell}, \gamma=\sqrt{\rho\epsilon}, \delta = \frac{d\ell}{\sqrt{\rho\epsilon}}e^{-\chi}$ and simply applying the definitions of $\mathfrak{G},\mathfrak{T},\mathfrak{F},g_{\text{thres}},t_{\text{thres}},f_{\text{thres}}$ into Lemma \ref{lem:main_saddle}. 
 \end{document}